\documentclass[a4paper,10pt]{article}

\usepackage{geometry}
\usepackage[english]{babel}
\usepackage{epsfig}
\usepackage{amsmath}
\usepackage{amssymb}
\usepackage{amsthm}
\usepackage{mathrsfs}
\usepackage{color}

\usepackage{amsthm}
\newtheorem{theorem}{Theorem}[section]
\newtheorem{definition}[theorem]{Definition}
\newtheorem{lemma}[theorem]{Lemma}
\newtheorem{proposition}[theorem]{Proposition}
\newtheorem{corollary}[theorem]{Corollary}
\newtheorem{remark}[theorem]{Remark}
\newtheorem{example}[theorem]{Example}
\newtheorem{examples}[theorem]{Examples}

\usepackage{amsfonts}

\newcommand{\oo}{{\mathbb{O}}}
\newcommand{\hh}{{\mathbb{H}}}

\newcommand{\cc}{{\mathbb{C}}}
\newcommand{\rr}{{\mathbb{R}}}
\newcommand{\zz}{{\mathbb{Z}}}
\newcommand{\nn}{{\mathbb{N}}}
\newcommand{\D}{\mathbb{D}}
\newcommand{\B}{\mathbb{B}}

\newcommand{\s}{{\mathbb{S}}}
\newcommand{\sr}{\mathcal{SR}}

\newcommand{\I}{\mathcal{I}}

\newcommand\vs[1]{{#1}_s^\circ}

\newcommand{\punto}{\bullet}
\newcommand{\stx}{\mathscr{S}}
\newcommand{\sto}{\mathrm{u}}
\newcommand{\Sto}{\mathrm{U}}
\newcommand\re{\operatorname{Re}}
\newcommand\im{\operatorname{Im}}
\newcommand\lra{\longrightarrow}
\newcommand{\ui}{\imath}
\newcommand{\OO}{\Omega}
\newcommand{\ord}{\mathrm{ord}}

\newcommand{\mr}{\mathrm}
\newcommand{\mc}{\mathcal}

\title{\bf Singularities of slice regular functions over real alternative $^*$-algebras}

\author{Riccardo Ghiloni and Alessandro Perotti \\ 
\small Dipartimento di Matematica, Universit\`a di Trento\\ 
\small Via Sommarive 14, I-38123 Povo Trento, Italy\\
\small riccardo.ghiloni@unitn.it, alessandro.perotti@unitn.it\\
\and
Caterina Stoppato
\\
\small Dipartimento di Matematica e Informatica ``U. Dini'', Universit\`a di Firenze \\
\small Viale Morgagni 67/A, I-50134 Firenze, Italy\\
\small stoppato@math.unifi.it}

\date{  }

\begin{document}

\maketitle


\begin{abstract}
The main goal of this work is classifying the singularities of \emph{slice regular functions} over a real alternative $^*$-algebra $A$. This function theory has been introduced in 2011 as a higher-dimensional generalization of the classical theory of holomorphic complex functions, of the theory of slice regular quaternionic functions launched by Gentili and Struppa in 2006 and of the theory of slice monogenic functions constructed by Colombo, Sabadini and Struppa since 2009. Along with this generalization step, the larger class of \emph{slice functions} over $A$ has been defined. We introduce here a new type of series expansion near each singularity of a slice regular function. This instrument, which is new even in the quaternionic case, leads to a complete classification of singularities. This classification also relies on some recent developments of the theory, concerning the algebraic structure and the zero sets of slice functions. Peculiar phenomena arise, which were not present in the complex or quaternionic case, and they are studied by means of new results on the topology of the zero sets of slice functions. The analogs of meromorphic functions, called \emph{(slice) semiregular} functions, are introduced and studied.
\end{abstract}


\section{Introduction}

In the theory of holomorphic and meromorphic functions of one complex variable, a general principle  states that such functions are in some sense determined by their zeros and singularities.
For this reason, any function theory whose aim is to generalize the classical theory of holomorphic functions must deal with zeros and singularities and must investigate their properties.

This is the case with the theory of \emph{slice regular} functions developed over the last ten years. It has been introduced in \cite{cras,advances} for quaternion-valued functions of a quaternionic variable and in \cite{rocky} for octonionic functions. The article \cite{clifford} considered functions over the Clifford algebra $C\ell_{0,3}=\rr_3$ and \cite{israel} treated the case of functions from $\rr^{n+1}$ to the Clifford algebra $C\ell_{0,n}=\rr_n$, defining the notion of \emph{slice monogenic} function. All these theories have been largely developed in subsequent literature: see the monographs \cite{librodaniele2,librospringer} and references therein. Finally, \cite{perotti} introduced a new approach that is valid on a large class of real alternative $^*$-algebras and comprises the aforementioned algebras.

While the present work is set on general alternative $^*$-algebras, for which we will give detailed preliminaries in Section~\ref{sec:preliminaries}, we can introduce the actors of the play by briefly overviewing the quaternionic case. Let $\s$ be the 2-sphere of  square roots of $-1$ in the quaternionic space $\hh$ and, for each $J \in \s$, let $\cc_J\simeq\cc$ be the commutative subalgebra of $\hh$ generated by $J$.
Let $\OO$ be an open subset of the quaternionic space $\hh$. A differentiable function $f:\OO\rightarrow\hh$ is called (Cullen or) \emph{slice regular} if, for each $J\in\s$, the restriction of $f$ to $\OO\cap\cc_J$ is holomorphic with respect to the complex structure defined by left multiplication by $J$. 

The quaternionic space $\hh$ can be decomposed as
$\hh=\bigcup_{J \in \s}\cc_J$, with  $\cc_J\cap \cc_K=\rr$ for every $J,K \in \s$ with $J \neq \pm K$.
Thanks to this ``slice'' nature of $\hh$, every slice regular function $f$ can be lifted to a unique holomorphic ``stem'' function $F:\rr\otimes\cc\rightarrow\hh\otimes\cc$, satisfying the relation 
\begin{equation}\label{lifting}
f(\phi_J(z))=\phi_J(F(z))
\end{equation}
 for each $z\in\rr\otimes\cc\simeq\cc$ and for all $J \in \s$. Here $\phi_J$ denotes the mapping sending $a+b\sqrt{-1}\in\hh\otimes\cc$ to $a+Jb\in\hh$.

This property is grounds for a generalized construction, with another alternative $^*$-algebra $A$ instead of $\hh$ and with an (appropriately defined) real algebraic subset $\s_A$ of $A$ instead of $\s$. An $A$-valued function $f$ that can be lifted to a function $F:\rr\otimes\cc\rightarrow A\otimes\cc$ by means of formula~\eqref{lifting} for all $J \in \s_A$ is called a \emph{slice function}. If, moreover, the lifted function $F$ is holomorphic, then $f$ is termed \emph{slice regular}. The effectiveness of this approach can be appreciated by looking at some applications, such as the Cauchy integral formula for slice functions \cite[Theorem~27]{perotti}, the definition of the (slice) functional calculus for normal operators in quaternionic Hilbert spaces \cite{QSFCalculus}, or the notion of spherical sectorial operators \cite{semigroups}.

The lifting property allows to endow slice functions with a rich algebraic structure, induced by that of $A\otimes\cc$. In particular, slice functions over $A$ form an alternative $^*$-algebra themselves, which is studied, along with the subalgebra of slice regular functions, in the recent work \cite{gpsalgebra}. As in the complex case, there is a rich interplay between the algebraic structure and the properties of the zeros. This study is applied and continued in the present paper, which treats singularities of slice regular functions.

After setting the grounds to work in Sections~\ref{sec:preliminaries} and~\ref{sec:series}, we introduce in Section~\ref{sec:power-laurent} a notion of Laurent series that generalizes~\cite{powerseries,expansionsalgebras,singularities}. The expansion into Laurent series allows to give a first classification of the singularities of slice regular functions: in Section~\ref{sec:types} we define \emph{removable singularities}, \emph{poles} and \emph{essential singularities}, as well as a notion of \emph{order}. However, there are some relevant differences with respect to the complex case. For instance, a pole of order $0$ is not necessarily a removable singularity for $f$. This is because the set of convergence of the Laurent expansion at a point $y$, completed with $y$ itself, is not always a neighborhood of $y$ within the domain of $f$.

To overcome this difficulty, in Sections~\ref{sec:spherical-laurent} and~\ref{sec:spherical-laurent-expansions} we construct a finer Laurent-type expansion, which is new even for the quaternionic case: the \emph{spherical Laurent expansion}. If $y=\alpha+\beta J$ with $\alpha,\beta \in \rr$ and $J \in \s_A$, we call the set $\s_y:=\alpha+\beta \s_A$ the \emph{sphere} through $y$. The spherical Laurent expansion at $y$ is valid in an open subset of the domain of definition of the function that is \emph{circular} (i.e., axially symmetric with respect to the real axis). 
The relation between the spherical Laurent expansion at $y$ and the Laurent expansions at $y$ and at the conjugate point $y^c$ is studied in Section~\ref{sec:relation} and then exploited in Section~\ref{sec:classification} to provide a complete classification of singularities. It turns out that the badness of a singularity is determined by its \emph{spherical order}, defined in terms of the spherical Laurent expansion. More precisely, we prove that, on each sphere having spherical order $2m$, the order is $m$ at each point, apart from a set of exceptions which must have lesser order. In particular, a singularity is removable if, and only if its spherical order is $0$. The study of how the order of a pole can vary over a sphere is completed in Section~\ref{sec:order}, by means of a new result on the topology of the zero sets of slice functions.

In the final Section~\ref{sec:algebra}, we introduce the analogs of meromorphic functions, called \emph{(slice) semiregular} functions. By means of the spherical Laurent expansions, we prove that semiregular functions over a circular domain form a real alternative $^*$-algebra. The so-called \emph{tame} elements of this algebra admit multiplicative inverses within the algebra.

\vspace{.5em}
{\bf Acknowledgements.} This work is supported by GNSAGA of INdAM and by the grants FIRB ``Differential Geometry and Geometric Function Theory" and PRIN ``Variet\`a reali e complesse: geometria, topologia e analisi armonica" of the Italian Ministry of Education.\\
We wish to thank the anonymous referee for her/his precious suggestions, which helped us improving our presentation.
\vspace{.5em}


\section{Assumptions and preliminaries}\label{sec:preliminaries}

In this section, we make some assumptions that will hold throughout the paper. Along with our assumptions, we recall the definition of alternative $^*$-algebra $A$ over the real field $\rr$, some of their properties and some finite-dimensional examples. Furthermore, we overview some material from~\cite{perotti,expansionsalgebras,gpsalgebra}; namely, the definition of two special classes of $A$-valued functions and some of their properties.

\vspace{.5em}

{\bf Assumption (A).} \emph{Let $A$ be an \emph{alternative algebra} over $\rr$; that is, a real vector space endowed with a bilinear multiplicative operation such that the associator $(x,y,z)=(xy)z-x(yz)$ of three elements of $A$ is an alternating function.}

\vspace{.5em}

The alternating property is equivalent to requiring that $(x,x,y) = 0 = (y,x,x)$ for all $x,y \in A$, which is automatically true if $A$ is associative. Another important property of alternative algebras is described by the following result of E. Artin (cf.~\cite[Theorem 3.1]{schafer}): the subalgebra generated by any two elements of $A$ is associative.

\vspace{.5em}

{\bf Assumption (B).} \emph{All algebras and subalgebras are assumed to be \emph{unitary}; that is, to have a multiplicative neutral element $1\neq0$. The field $\rr$ of real numbers is identified with the subalgebra generated by $1$.}

\vspace{.5em}

With this notation, $\rr$ is always included in the \emph{nucleus} $\mathfrak{N}(A):=\{r \in A\, |\, (r,x,y)=0\ \forall\, x,y \in A\}$ and in the \emph{center} $\mathfrak{C}(A):=\{r \in \mathfrak{N}(A)\, |\, rx=xr\ \forall\, x\in A\}$ of the algebra $A$. We refer to~\cite{schafer} for further details on alternative algebras.

\vspace{.5em}

{\bf Assumption (C).} \emph{$A$ is assumed to be a $^*$-algebra. In other words, it is assumed to be equipped with a \emph{$^*$-involution} $x \mapsto x^c$; that is, a (real) linear transformation of $A$ with the following properties: $(x^c)^c=x$ for every $x \in A$, $(xy)^c=y^cx^c$ for every $x,y \in A$ and $x^c=x$ for every $x \in \rr$.}

\vspace{.5em}

For every $x \in A$, the \emph{trace} $t(x)$ of $x$ and the (squared) \emph{norm} $n(x)$ of $x$ are defined as
\begin{equation}\label{traceandnorm}
t(x):=x+x^c \quad \text{and} \quad n(x):=xx^c.
\end{equation}
If we set
\begin{equation} \label{eq:s_A}
\s_A=\{x \in A \, | \, t(x)=0, n(x)=1\},
\end{equation}
then the subalgebra $\cc_J$ generated by any $J \in \s_A$ is isomorphic to the complex field. Observe that $t(x)=0$ consist of real linear equations and 
$n(x)=1$ consists of real quadratic equations. Now, if we define the \emph{quadratic cone} of $A$ as
\[
Q_A:=\rr \cup \big\{x \in A \, \big| \, t(x) \in \rr ,n(x) \in \rr, t(x)^2<4n(x)\big\}
\]
then
\begin{equation} \label{eq:slice}
\textstyle
\text{$Q_A=\bigcup_{J \in \s_A}\cc_J$.}
\end{equation}
Moreover, $\cc_I \cap \cc_J=\rr$ for every $I,J \in \s_A$ with $I \neq \pm J$. In other words, every element $x$ of $Q_A \setminus \rr$ can be written as follows: $x=\alpha+\beta J$, where $\alpha \in \rr$ is uniquely determined by $x$, while $\beta \in \rr$ and $J \in \s_A$ are uniquely determined by $x$, but only up to sign. If $x \in \rr$, then $\alpha=x$, $\beta=0$ and $J$ can be chosen arbitrarily in $\s_A$. Therefore, it makes sense to define the \emph{real part} $\re(x)$ and the \emph{imaginary part} $\im(x)$ by setting $\re(x):=t(x)/2=(x+x^c)/2$ and $\im(x):=x-\re(x)=(x-x^c)/2$. Finally, for all $J \in \s_A$, we have that $J^c=-J$ so that the isomorphism between $\cc_J$ and $\cc$ is also a $^*$-algebra isomorphism. In particular, if $x=\alpha+\beta J$ for some $\alpha,\beta \in \rr$ and $J \in \s_A$, then $x^c=\alpha-\beta J$ and $n(x) = n(x^c)= \alpha^2+\beta^2$. Hence, for such an $x$ we may write $|x|:=\sqrt{n(x)}=\sqrt{\alpha^2+\beta^2}$ just as we would for a complex number. We refer the reader to~\cite[\S 2]{perotti} for a proof of the preceding assertions.

The function theory we will refer to in this paper involves $A$-valued functions with domains contained in the quadratic cone $Q_A$. Therefore, it is natural to take the next assumption.

\vspace{.5em}

{\bf Assumption (D).} \emph{The real dimension of $A$ is assumed to be finite. The quadratic cone $Q_A$ is assumed to strictly contain $\rr$; equivalently, the real algebraic subset $\s_A$ of $A$ is assumed not to be empty.}

\vspace{.5em}

Since left multiplication by an element of $\s_A$ induces a complex structure on $A$, we infer that the real dimension of $A$ equals $2h+2$ for some non-negative integer $h$. Then the real alternative *-algebra $A$ has the following useful splitting property, \cite[Lemma 2.3]{expansionsalgebras}: for each $J \in \s_A$, there exist $J_1,\ldots,J_h \in A$ such that $\{1,J,J_1,JJ_1,\ldots,J_h,JJ_h\}$ is a real vector basis of $A$, called a \emph{splitting basis} of $A$ associated with $J$. 

\vspace{.5em}

{\bf Assumption (E).} \emph{Let $A$ be equipped with its natural Euclidean topology and differentiable structure as a finite dimensional real vector space.}

\vspace{.5em}
The relative topology on each $\cc_J$ with $J \in \s_A$ clearly agrees with the topology determined by the natural identification between $\cc_J$ and $\cc$.
Given a subset $E$ of $\cc$, its \emph{circularization} $\OO_E$ is defined as the following subset of $Q_A$:
\[
\OO_E:=
\left\{x \in Q_A \, \big| \, \exists \alpha,\beta \in \rr, \exists J \in \s_A \mathrm{\ s.t.\ } x=\alpha+\beta J, \alpha+\beta i \in E\right \}.
\]
A subset of $Q_A$ is termed \emph{circular} if it equals $\OO_E$ for some $E\subseteq\cc$. For instance, given $x=\alpha+\beta J \in Q_A$ we have that
\[
\s_x:=\alpha+\beta \, \s_A=\{\alpha+\beta I \in Q_A \, | \, I \in \s_A\}
\]
is circular, as it is the circularization of the singleton $\{\alpha+i\beta\}\subseteq \cc$. We observe that $\s_x=\{x\}$ if $x \in \rr$. On the other hand, for $x \in Q_A \setminus \rr$, the set $\s_x$ is obtained by real translation and dilation from $\s_A$. Such sets are called \emph{spheres}, because the theory has been first developed in the special case of division algebras $A=\cc,\hh,\oo$ where they are genuine Euclidean spheres (see the forthcoming Examples~\ref{ex:divisionalgebras}).

The class of functions we consider was defined in~\cite{perotti} by means of the complexified algebra $A_{\cc}=A \otimes_{\rr} \cc=\{x+\ui y \, | \, x,y \in A\}$ of $A$, endowed with the following product:
\[
(x+\ui y)(x'+\ui y')=xx'-yy'+\ui (xy'+yx').
\]
The algebra $A_{\cc}$ is still alternative: the equality $(x+\ui y, x+\ui y, x'+\ui y') =0$ can be proven by direct computation, taking into account that $(x,y,x') = - (y,x,x')$ and  $(x,y,y') = - (y,x,y')$ because $A$ is alternative. Moreover, $A_{\cc}$ is associative, if and only if $A$ is. We shall denote by $\rr_\cc$ the real subalgebra $\rr+\ui\rr\simeq\cc$ of $A_\cc$, which is included in the center of $A_\cc$.
In addition to the complex conjugation $\overline{x+\ui y}=x-\ui y$, we may endow $A_\cc$ with a $^*$-involution $x+\ui y \mapsto (x+\ui y)^c := x^c+\ui y^c$, which makes it a $^*$-algebra.

\begin{definition} \label{def:slice-function}
Let $D$ be a nonempty subset of $\cc$, invariant under the complex conjugation $z=\alpha+i\beta \mapsto \overline{z}=\alpha-i\beta$. A function $F=F_1+\ui F_2:D \lra A_\cc$ with $A$-components $F_1$ and $F_2$ is called a \emph{stem function} on $D$ if $F(\overline{z})=\overline{F(z)}$ for every $z \in D$ or, equivalently, if $F_1(\overline{z})=F_1(z)$ and $F_2(\overline{z})=-F_2(z)$ for every $z \in D$.
If $\OO := \OO_D$ then a function $f:\OO \lra A$ is called a \emph{(left) slice function} if there exists a stem function  $F=F_1+\ui F_2:D \lra A_\cc$ such that, for all $z=\alpha+i\beta \in D$ (with $\alpha,\beta \in \rr$), for all $J \in \s_A$ and for $x=\alpha+\beta J$,
\[
f(x)=F_1(z)+JF_2(z),
\]
In this situation, we say that $f$ is induced by $F$ and we write $f=\I(F)$. If $F_1$ and $F_2$ are $\rr$-valued, then we say that the slice function $f$ is \emph{slice preserving}.
We denote by $\mc{S}(\OO)$ the real vector space of slice functions on $\OO$ and by $\mc{S}_\rr(\OO)$ the subspace of slice preserving functions.
\end{definition}

The terminology used in the definition is justified by the following property, valid whenever the cardinality of $\s_A$ is greater than $2$: the components $F_1$ and $F_2$ of a stem function $F$ are $\rr$-valued if and only if the slice function $f=\I(F)$ maps every ``slice'' $\OO_J:=\OO\cap\cc_J$ into $\cc_J$ (cf.~\cite[Proposition~10]{perotti}).

The algebraic structure of $\mc{S}(\OO)$ can be described as follows, see \cite[\S 2]{gpsalgebra}.

\begin{proposition}\label{prop:algebra}
The stem functions $D \to A_\cc$ form an alternative $^*$-algebra over $\rr$ with pointwise addition $(F+G)(z) = F(z)+G(z)$, multiplication $(FG)(z) = F(z)G(z)$ and conjugation $F^c(z) = F(z)^c$.
Besides the pointwise addition $(f,g) \mapsto f + g$, there exist unique operations of multiplication $(f,g) \mapsto f \cdot g$ and conjugation $f \mapsto f^c$ on the set $\mc{S}(\OO)$ of slice functions on $\OO:=\OO_D$ such that the mapping $\I$ is a $^*$-algebra isomorphism from the $^*$-algebra of stem functions on $D$ to $\mc{S}(\OO)$. The product $f\cdot g$ of two functions $f,g \in \mc{S}(\OO)$ is called \emph{slice product} of $f$ and $g$. 
\end{proposition}

In general, $(f \cdot g)(x) \neq f(x)g(x)$. To compute $f \cdot g$ with $f=\I(F),g=\I(G) \in \mc{S}(\OO)$, one needs instead to compute $FG$ and then $f \cdot g=\I(FG)$. Similarly, $f=\I(F)$ implies $f^c=\I(F^c)$. The \emph{normal function} of $f$ in $\mc{S}(\OO)$ is defined as
\[
N(f)=f \cdot f^c=\I(FF^c).
\]
If $f$ is slice preserving, then we have that $f \cdot g=fg=g \cdot f$, that $f=f^c$ and that $N(f)=f^2$. 

\begin{example}\label{ex:Delta}
For fixed $y \in Q_A$, the binomial $f(x) := x-y$ is a slice function induced by $F(\alpha+i\beta):=\alpha-y + \ui \beta$. The slice product $(f \cdot f)(x) = x^2-2xy+ y^2$ does not coincide with the pointwise square $f(x)^2 = x^2 -xy -yx +y^2$ if $y$ does not commute with every element of $Q_A$.
The conjugate function is $f^c(x) = x-y^c$ and its normal function $N(f)(x) = (x-y) \cdot (x-y^c)$ coincides with the slice preserving quadratic polynomial
\[
\Delta_y(x):=x^2-xt(y)+n(y).
\]
If $y' \in Q_A$, then $\Delta_{y'}=\Delta_y$ if and only if $\s_{y'}=\s_y$.
\end{example}

Let $\OO:=\OO_D$ and let $f:\OO \lra A$ be a slice function. Given $y=\alpha+\beta J$ and $z=\alpha+\beta K$ in $\OO$ for some $\alpha,\beta \in \rr$ and $J,K \in \s_A$ with $J-K$ invertible, as a direct consequence of the definition of slice function, the following representation formula holds for all $x=\alpha+\beta I$ with $I \in \s_A$:

\begin{equation}\label{rep1}
f(x)=(I-K)\left((J-K)^{-1}f(y)\right)-(I-J)\left((J-K)^{-1}f(z)\right).
\end{equation}
In particular, for $K=-J$,
\begin{equation}\label{rep2}
f(x)=\frac12\left(f(y)+f(y^c)\right)-\frac{I}2\left(J\left(f(y)-f(y^c)\right)\right).
\end{equation}

Suppose $f=\I(F) \in \mc{S}(\OO)$ with $F=F_1+\ui F_2$. It is useful to define a function $\vs f:\OO \lra A$, called \emph{spherical value} of $f$, and a function $f'_s:\OO \setminus \rr \lra A$, called \emph{spherical derivative} of $f$, by setting
\[
\vs f(x):=\frac{1}{2}(f(x)+f(x^c))
\quad \text{and} \quad
f'_s(x):=\frac{1}{2}\im(x)^{-1}(f(x)-f(x^c)).
\] 
The notations originally used for these functions were, $v_s f$ and $\partial_s f$, respectively. $\vs f$ and $f'_s$ are slice functions and they are constant on each sphere $\s_x\subseteq \OO$.

Within the class of slice functions, let us consider a special subclass having nice properties that recall those of holomorphic functions of a complex variable. We begin with the next remark, where we take advantage of the fact that, for $\OO:=\OO_D$ and for each $J \in \s_A$, the slice $\OO_J:=\OO\cap\cc_J$ is equivalent to $D$ under the natural identification between $\cc_J$ and $\cc$.

\begin{remark}
If $\OO:=\OO_D$ is an open subset of $Q_A$, then each slice $\OO_J=\OO\cap\cc_J$ with $J \in \s_A$ is open in the relative topology of $\cc_J$; therefore, $D$ itself is open in $\cc$. The continuous and $\mathscr{C}^1$ stem functions on $D$ form $^*$-subalgebras of the $^*$-algebra of stem functions. Their respective images through $\I$, which we may denote as $\mc{S}^0(\OO)$ and $\mc{S}^1(\OO)$, are $^*$-subalgebras of $\mc{S}(\OO)$.
\end{remark}

Now take $f = \I(F) \in \mc{S}^1(\OO)$. The derivative $\partial F/\partial \overline{z}:D \lra A_{\cc}$ with respect to $\overline{z}= \alpha - \beta\ui$, that is,
\[
\frac{\partial F}{\partial\overline{z}}:=\frac{1}{2}\left(\frac{\partial F}{\partial\alpha}+\ui\frac{\partial F}{\partial\beta}\right)
\]
and the analogous $\partial F/\partial z:D \lra A_{\cc}$ are still stem functions, which induce the slice functions $\partial f/\partial x^c:=\I(\partial F/\partial \overline{z})$ and  $\partial f/\partial x:=\I(\partial F/\partial z)$ on $\OO$.

\begin{definition} \label{def:slice-regularity} 
Let $\OO:=\OO_D$ be open in $Q_A$. A slice function $f \in \mc{S}^1(\OO)$ is called \emph{slice regular} if $\partial f/\partial x^c\equiv0$ in $\OO$. We denote by $\mc{SR}(\OO)$ the set of slice regular functions on $\OO$.
\end{definition}

Equivalently, the set $\mc{SR}(\OO)$ is the $^*$-subalgebra of $\mc{S}(\OO)$ induced via the mapping $\I$ by the $^*$-subalgebra of holomorphic stem functions on $D$. The most classic examples of slice regular functions are polynomials.

\begin{example}
Every polynomial of the form $\sum_{m=0}^n x^ma_m = a_0+x a_1 + \ldots + x^n a_n$ with coefficients $a_0, \ldots, a_n \in A$ is a slice regular function on the whole quadratic cone $Q_A$.
\end{example}

Slice regularity is naturally related to complex holomorphy in the following sense, see \cite[Lemma 2.4]{expansionsalgebras}.
    
\begin{lemma}\label{splitting} 
Let $\OO:=\OO_D$ be open in $Q_A$.
Let $J \in \s_A$ and let $\{1,J,J_1,JJ_1,\ldots,J_h,JJ_h\}$ be a splitting basis of $A$ associated with $J$. For $f \in \mc{S}^1(\OO)$, let $f_0,f_1,\ldots,f_h : \OO_J \to \cc_J$ be the $\mathscr{C}^1$ functions such that $f_{|_{\OO_J}}=\sum_{\ell=0}^h f_\ell J_{\ell}$, where $J_0:=1$. Then $f$ is slice regular if, and only if, for each $\ell \in \{0,1,\ldots,h\}$, $f_\ell$ is holomorphic from $\OO_J$ to $\cc_J$, both equipped with the complex structure associated with left multiplication by $J$.
\end{lemma}

For future reference, let us make a further remark concerning the domains of our functions.
Let $\OO:=\OO_D$ with $D\subseteq \cc$ preserved by complex conjugation. If $\OO$ is open in $Q_A$, then $D$ can be decomposed into a disjoint union of open subsets of $\cc$, each of which either
\begin{enumerate}
\item intersects the real line $\rr$, is connected and preserved by complex conjugation; or
\item does not intersect $\rr$ and has two connected components, $D^+$ in the open upper half-plane $\cc^+$ and $D^-$ in the open lower half-plane $\cc^-$, switched by complex conjugation.
\end{enumerate}
Therefore, when $D$ is open in $\cc$, without loss of generality it falls within case 1 or case 2. In the former case, the resulting domain $\OO$ is called a \emph{slice domain} because each slice $\OO_J$ with $J \in \s_A$ is a domain in the complex analytic sense (more precisely, it is an open connected subset of $\cc_J$). In case 2, we will call $\OO$ a \emph{product domain} as it is homeomorphic to the product between the complex domain $D^+$ and the sphere $\s_A$. In such a case each slice $\OO_J$ has two connected components, namely $\OO_J^+ = \{\alpha + \beta J \in Q_A\,|\, \alpha + i \beta \in D^+\}$ and the analogous $\OO_J^-$.

In the sequel, every mention of a slice function on a set $\OO$ will automatically imply that $\OO:=\OO_D$ for some nonempty $D \subseteq \cc$, invariant under complex conjugation and not necessarily open.

We now recall a few results from \cite{gpsalgebra} concerning the existence of a multiplicative inverse for a slice function.

\begin{theorem}\label{invertibles_f}
A slice function $f$ admits a multiplicative inverse $f^{-\punto}$ if, and only if, $f^c$ does. If this is the case, then $(f^c)^{-\punto} = (f^{-\punto})^c$. Furthermore, $f$ admits a multiplicative inverse $f^{-\punto}$ if, and only if, both $N(f)$ and $N(f^c)$ do. If this is the case, then: 
\[
f^{-\punto} =f^c\cdot N(f)^{-\punto} = N(f^c)^{-\punto}\cdot f^c.
\]
Moreover, $N(f)^{-\punto} = (f^{-\punto})^c\cdot f^{-\punto} = N((f^{-\punto})^c)$.
\end{theorem}

Theorem~\ref{invertibles_f} is particularly useful for the following class of functions.

\begin{definition}\label{def:tame}
A slice function is termed \emph{tame} if $N(f)$ is slice preserving and it coincides with $N(f^c)$.
\end{definition}

Indeed, Theorem~\ref{invertibles_f} yields the next proposition, where we will use the notation
\[
V(g):=\{x \in \OO \, | \, g(x)=0\}
\]
for all $g \in \mc{S}(\OO)$.

\begin{proposition}\label{reciprocal}
Let $f \in \mc{S}(\OO)$ be tame. If $\OO' := \OO \setminus V(N(f))$ is not empty, then $f$ admits a multiplicative inverse in $\mc{S}(\OO')$, namely
\[
f^{-\punto}(x) = (N(f)^{-\punto} \cdot f^c) (x)= (N(f)(x))^{-1} f^c(x).
\]
For any tame $g \in \mc{S}(\OO')$,
\begin{equation}\label{Nfg}
N(f \cdot g) = N(f) N(g) = N(g) N(f) = N(g\cdot f)
\end{equation}
and $f\cdot g$ is a tame element of $\mc{S}(\OO'')$ with $\OO'' := \OO' \setminus V(N(g))$ (provided $\OO''$ is not empty).

Furthermore, if $\OO'$ is open then $f^{-\punto}$ is slice regular if and only if $f$ is slice regular in $\OO'$. Finally, if $\OO$ is open in $Q_A$ and $f \in \mc{SR}(\OO)$ has $N(f)\not \equiv 0$, then $V(N(f))$ is the circularization of a closed and discrete subset of $\cc$.
\end{proposition}

\begin{example}\label{ex:DeltaInv}
For fixed $y \in Q_A$, consider the binomial $f(x) := x-y$. The functions $N(f)$ and $N(f^c)$ coincide with the slice preserving function $\Delta_y(x) = x^2-xt(y)+n(y)$, whose zero set is $\s_y$. By Proposition~\ref{reciprocal}, $f$ admits a multiplicative inverse
\[
f^{-\punto}(x) = \Delta_y(x)^{-1} (x-y^c) = (x^2-xt(y)+n(y))^{-1}(x-y^c)
\]
in $\mc{S}(\OO')$, where $\OO' = Q_A \setminus \s_y$. In the sequel, for each $n \in \zz$ we will denote by $(x-y)^{\punto n}$ the $n^{\rm{th}}$-power of $f$ with respect to the slice product. For the power $(x-y)^{\punto (-n)}$ we might also use the notation $(x-y)^{-\punto n}$.
\end{example}

More in general, we will be using the notation $f(x)\cdot g(x)$ for $(f\cdot g)(x)$ and the notation $f(x)^{\punto n}$ for $f^{\punto n}(x)$ when they are unambiguous. In such a case, $x$ will necessarily stand for a variable. We recall that $n(x):=xx^c$ and that a $^*$-algebra is called \emph{nonsingular} when for every element $x$ the equality $n(x) = 0$ implies $x=0$.

\begin{proposition}\label{SRnonsingular}
Assume that $\OO$ is open. The $^*$-algebra $\mc{SR}(\OO)$ is nonsingular if and only if $A$ is nonsingular and $\OO$ is a union of slice domains.
\end{proposition}

Since $N(f)$ is the norm of $f$ in the $^*$-algebra $\mc{SR}(\OO)$, when the latter is nonsingular we have that $N(f) \equiv 0$ implies $f \equiv 0$. The next consequences of Propositions~\ref{reciprocal} and~\ref{SRnonsingular} will also prove useful in the sequel.

\begin{proposition}\label{tameregularproduct}
Assume that $\OO$ is a slice domain or a product domain. If $f$ and $g$ are slice regular and tame in $\OO$, and $N(f)$ and $N(g)$ do not vanish identically, then $N(f\cdot g) = N(f)N(g) = N(g)N(f)$ in $\OO$. As a consequence, $f\cdot g$ is a tame element of $\mc{SR}(\OO)$.
\end{proposition}

\begin{proposition}\label{tamenotzerodivisor}
If $A$ is nonsingular and $\OO$ is a slice domain, then each element $f\in\mc{SR}(\OO)$ that is tame cannot be a zero divisor in $\mc{S}^0(\OO)$.
\end{proposition}

We now take an additional assumption that allows, as explained in~\cite{expansionsalgebras}, to consider power series as further examples of slice regular functions.

\vspace{.5em}

{\bf Assumption (F).} 
\emph{We assume that our real alternative *-algebra $A$ is equipped with a norm $\|\cdot\|_A$ such that
\begin{equation} \label{eq:norm}
\|x\|_A^2= n(x)
\end{equation}
for every $x \in Q_A$.
}

\vspace{.5em}

Since the dimension of $A$ is assumed to be finite, the topology induced by $\|\cdot\|_A$ on $A$ is equal to the Euclidean one. For each $J \in \s_A$, the norm $\|\cdot\|_A$ coincides in $\cc_J$ with the modulus function $|\cdot|$ induced by the natural identification of $\cc_J$ with the complex field. When appropriate, we will also use the latter notation. Assumption (F) guarantees that $\s_A$, defined by formula~\eqref{eq:s_A}, is a closed subset of the compact set
\begin{equation*}
S_A:=\big\{x \in A \, \big| \, \|x\|_A =1 \big\}.
\end{equation*}
Hence, $\s_A$ is compact and so is
\begin{equation*}
S_A \cap Q_A =  \{\alpha+\beta J \in A \, | \,  (\alpha,\beta) \in S^1, J \in \s_A\}.
\end{equation*}
In particular, each sphere $\s_x$ with $x \in Q_A$ is compact and $Q_A$ is a closed subset of $A$.

\begin{remark}\label{multiplicativeestimates}
If we set
\[
C_A:=\max_{x,y \in S_A} \|x y\|_A \in [1,+\infty), \qquad
c_A:= \min_{x,z \in S_A \cap Q_A, y \in S_A}\|(x y) z\|_A \in (0,1].
\]
then
\begin{equation} \label{eq:C_A}
\|x y\|_A \leq C_A\, \|x\|_A\, \|y\|_A \; \text{ for every $x,y \in A$} 
\end{equation}
and 
\begin{equation} \label{eq:c_A}
c_A\, \|x\|_A\, \|y\|_A \leq \|x y\|_A \; \text{ for every $x, y \in A$ with $x \in Q_A$ or $y \in Q_A$}.
\end{equation}
\end{remark}

The previous remark allows to evaluate the convergence of a power series with the root criterion, which leads to another class of examples of slice regular functions.

\begin{example}
Every power series of the form $\sum_{n \in \nn} x^n a_n$ with $\{a_n\}_{n \in \nn}\subseteq A$ converges on the intersection between $Q_A$ and $B(0, R) = \{x \in A\ |\ \Vert x \Vert_A<R\}$ for some $R \in [0,+\infty] := [0,+\infty) \cup \{+\infty\}$. If $R>0$, then the sum of the series is a slice regular function on $Q_A \cap B(0,R)$.
\end{example}

We now exhibit examples of algebras fulfilling all our assumptions (A) - (F).

\begin{examples}\label{ex:divisionalgebras}
The complex field $\cc$, the skew field of quaternions $\hh$, and the algebra $\oo$ of octonions (see~\cite{ebbinghaus} for their definitions) are unitary alternative algebras of dimensions $2, 4$ and $8$, respectively. On all such algebras $A$, a $^*$-involution is defined to act as $(r+v)^c=r-v$ for all $r \in \rr$ and all $v$ in the Euclidean orthogonal complement of $\rr$ in $A$. The norm $n(x)$ turns out to coincide with the squared Euclidean norm $\Vert x\Vert^2$. As a consequence, $Q_A = A$ and the respective $\s_A$ coincide with the unit spheres $S^0,S^2,S^6$ in the Euclidean subspaces $\im(\cc),\im(\hh),\im(\oo)$. Finally, each such $A$ is a \emph{division algebra}: every $x\neq0$ admits a multiplicative inverse, namely $x^{-1} = n(x)^{-1} x^c = x^c\, n(x)^{-1}$. In other words, $A \setminus\{0\}$ is a multiplicative Moufang loop (a multiplicative group in the associative cases $A=\cc,\hh$).
\end{examples}

Slice regularity coincides with holomorphy if $A=\cc$. Furthermore, it had originally been introduced with a completely different approach in the case $A= \hh$ (see~\cite[Chapter 1]{librospringer}, which also points out the original references) and $A = \oo$ (see~\cite{rocky}).

Another important class of algebras fulfilling our assumptions is the following.

\begin{examples}
For all $m\geq1$, the Clifford algebra $\rr_m = C\ell_{0,m}$ is a unitary associative algebra of dimension $2^m$ with the following conventions:
\begin{itemize}
\item $1, e_1,\ldots,e_m, e_{12}, \ldots, e_{m-1,m}, e_{123},\ldots,e_{1\ldots m}$ denotes the standard basis; 
\item $1$ is defined to be the neutral element;
\item $e_i^2 := -1$ for all $i \in \{0,\ldots,m\}$;
\item $e_ie_j = -e_j e_i$ for all distinct $i,j$;
\item for all $i_1,\ldots,i_s \in \{1,\ldots,m\}$ with $i_1<\ldots<i_s$, the product $e_{i_1}\ldots e_{i_s}$ is defined to be $e_{i_1\ldots i_s}$.
\end{itemize}
The algebra $\rr_m$ becomes a $^*$-algebra when endowed with Clifford conjugation $x\mapsto x^c$, defined to act on $e_{i_1\ldots i_s}$ as the identity $id$ if $s\equiv0,3 \mod 4$ and as $-id$ if $s\equiv1,2 \mod 4$.

If we denote by $\Vert x\Vert$ and $\langle \cdot,\cdot \rangle$ the Euclidean norm and scalar product in $\rr_m$, it holds
\begin{equation}\label{cliffordnorm}
n(x) = \sum_{s\, \equiv\, 0,3 \text{ $\mathrm{mod}$ }4} \langle x, e_{i_1\ldots i_s}x \rangle e_{i_1\ldots i_s} = \Vert x\Vert^2 + \langle x, e_{123}x \rangle e_{123} + \ldots
\end{equation}
The quadratic cone $Q_{\rr_m}$ is a real algebraic subset of $\rr_m$, whose dimension is $2,4,6$ for $m = 1,2,3$ and then grows exponentially with $m$. Since the norm $n(x)$ of each $x = \alpha + \beta J \in Q_{\rr_m}$ is the non negative real number $n(x)=\alpha^2+\beta^2$, we infer at once that $n$ coincides with $\Vert \cdot \Vert^2$ in $Q_{\rr_m}$.

Another consequence of formula~\eqref{cliffordnorm} is that $\rr_m$ is nonsingular. On the other hand, $\rr_m$ is not a division algebra for $m\geq 3$ since for all nonzero $q \in \rr_2 \subseteq \rr_m$, the numbers $q\pm qe_{123}$ are \emph{zero divisors}; that is, nontrivial solutions $x$ to $ax=0$ or $xa=0$ with $a \neq 0$.

For more details, see~\cite{GHS}. See also~\cite[Example 1.15]{gpsalgebra}.
\end{examples}

As we mentioned in the introduction, the theory of $\rr_3$-valued slice regular functions had been introduced in~\cite{clifford}. Moreover,~\cite{israel} had introduced the related theory of \emph{slice monogenic} functions from $\rr^{m+1}$ to $\rr_m$.

\begin{example}\label{bicomplex}
The algebra of bicomplex numbers $\B\cc = \cc \oplus \cc$, with 
\[(z_1,z_2)\,(w_1,w_2) := (z_1w_1,z_2w_2),\]
is an associative and commutative algebra over $\rr$. It is unitary, with $1_{\B\cc} = (1,1)$, and it is a $^*$-algebra when endowed with the involution $(z_1,z_2)^c:=(z_1^c,z_2^c)$. Clearly, 
\[t(z_1,z_2) = (t(z_1),t(z_2)) = 2 (Re(z_1),Re(z_2)),\]
\[n(z_1,z_2) = (n(z_1),n(z_2)) = (|z_1|^2,|z_2|^2).\]
In particular, $\B\cc$ is nonsingular. Moreover, it holds $\s_{\B\cc} = \{e^+,e^-,-e^+,-e^-\}$ with $e^+:=(i,i)$ and $e^-:=(i,-i)$. Hence, $Q_{\B\cc}$ is the union of the planes $\cc_{e^+} = \{(z_1,z_1)\,|\, z_1 \in \cc\}$ and $\cc_{e^-} = \{(z_1,z_1^c)\,|\, z_1 \in \cc\}$, which intersect at $\rr = \rr\,1_{\B\cc}$.
If we use the rescaled Euclidean norm
\[\|(z_1,z_2)\|_{\B\cc} := \sqrt{\frac{|z_1|^2+|z_2|^2}2}\]
then the function $n$ coincides with $\|\cdot\|_{\B\cc}^2 = \|\cdot\|_{\B\cc}^2\,1_{\B\cc}$ in $Q_{\B\cc}$.
\end{example}

Other known algebras, such as that of split-complex numbers $\s\cc = C\ell_{1,0}$, do not fulfill Assumption (D). The algebras of split-quaternions $\s\hh= C\ell_{1,1}$, of dual quaternions $\D\hh$ and of split-octonions $\s\oo$ with their standard conjugations do not fulfill Assumption (F). Indeed, in each of these three $^*$-algebras $A$ the sphere $\s_A$ is a real algebraic subset that is not compact. See~\cite[Examples 1.13]{gpsalgebra} for more details.


\section{Convergence of series of continuous slice functions}\label{sec:series}

In this section, we set up the basic framework to deal with series of continuous slice functions. This will ease the subsequent construction of Laurent series and spherical Laurent series over $A$, which will in turn allow us to classify and study the singularities of slice regular functions. 

By \cite[Proposition 7(1)]{perotti}, a slice function $f \in \mc{S}(\OO)$ is continuous if, and only if, it belongs to the subalgebra $\mc{S}^0(\OO)$ of slice functions on $\OO$ that are induced by continuous stem functions.

\begin{definition}
Let $\OO$ be a circular open subset of $Q_A$ and let $\{f_n\}_{n \in \zz}$ be a sequence in $\mc{S}^0(\OO)$. The series $f=\sum_{n \in \zz} f_n$ is termed \emph{totally convergent} on a circular compact subset $T$ of $\OO$ if the number series
\[\sum_{n \in \zz} \max_{x \in T}\|f_n(x)\|_A\]
converges. A nonempty circular open subset $\OO'$ of $\OO$ is called a \emph{domain of convergence} for $f$ if:
\begin{itemize}
\item $f$ converges totally on each circular compact subset of $\OO'$; and
\item every circular open subset $\OO''$ of $\OO$ with $\OO'\subsetneq\OO''$ includes a point $x$ where $\sum_{n \in \nn} f_n(x)$ or $\sum_{m \in \nn} f_{-m}(x)$ does not converge with respect to $\|\cdot\|_A$.
\end{itemize}
If no such $\OO'$ exists, then we say that the domain of convergence of $f(x)$ is empty.
\end{definition}

\begin{remark}
Let $\OO$ be a circular open subset of $Q_A$. Let $\{f_n\}_{n \in \zz}$ be a sequence in $\mc{S}^0(\OO)$ and suppose that the series $f=\sum_{n \in \zz} f_n$ has a nonempty domain of convergence $\OO'$. Then for every $y \in \OO'$, since the sphere $\s_y$ is circular and compact under our Assumption (F), the series $f$ converges at $y$. 
As a consequence, $f$ admits a unique domain of convergence (possibly empty). Therefore, we may refer to it as \emph{the} domain of convergence of $f$.
\end{remark}

The next remark is an immediate consequence of~\cite[Lemma 3.2]{global} and of Lemma~\ref{splitting}.

\begin{remark}\label{regularsum}
Let $\{f_n\}_{n \in \zz}$ be a sequence in $\mc{S}^0(\OO)$, where $\OO$ is a circular open subset of $Q_A$. If $f=\sum_{n \in \zz} f_n$ converges totally on each circular compact $T\subseteq\OO$, then its sum is a continuous slice function $f \in \mc{S}^0(\OO)$. If each $f_n$ is slice regular and if $f$ has a nonempty domain of convergence $\OO'\subseteq\OO$, then the sum is a slice regular function $f \in \mc{SR}(\OO')$.
\end{remark}

We are ready to prove the next result, which exploits again our Assumption (F). For each $J \in \s_A$, we will use the notation $\mathscr{H}_J:=\{\alpha + \beta J\in\cc_J\ |\ \alpha, \beta \in \rr, \beta\geq 0\}$.

\begin{theorem}\label{totalcircular}
Let $\{f_n\}_{n \in \zz}$ be a sequence of continuous slice functions on a circular open subset $\OO$ of $Q_A$ and let $T$ be a circular compact subset of $\OO$. The following assertions are equivalent.
\begin{enumerate}
\item $\sum_{n \in \zz} f_n$ converges totally in $T$.
\item There exists $J \in \s_A$ such that $\sum_{n \in \zz} f_n$ converges totally in $T_J$.
\item There exist $J,K \in \s_A$ with $J-K$ invertible such that $\sum_{n \in \zz} f_n$ converges totally in $T\cap\mathscr{H}_J$ and in $T\cap\mathscr{H}_K$.
\end{enumerate}
\end{theorem}

\begin{proof}
Clearly, {\it 1} implies both {\it 2} and {\it 3}.
Moreover, {\it 2} implies {\it 1} as a consequence of the following estimate, which follows from the Representation Formula~\eqref{rep2} and from Remark~\ref{multiplicativeestimates}: if $x = \frac 12 (z+z^c) + \frac I 2 \left(J (z^c- z)\right)$ with $z \in T_J$, then
\begin{align*}
\|f_n(x)\|_A &\leq \frac12 \|f_n(z) + f_n(z^c)\|_A + \frac12 \|I(J(f_n(z) - f_n(z^c)))\|_A\\
&\leq \frac12 \|f_n(z) + f_n(z^c)\|_A + \frac{C_A^2}2 \|f_n(z) - f_n(z^c)\|_A\\
&\leq \frac{1+C_A^2}2 (\|f_n(z)\|_A + \|f_n(z^c)\|_A).
\end{align*}
A similar estimate, based on the Representation Formula~\eqref{rep1}, proves that {\it 3} implies {\it 1}.
\end{proof}

\begin{corollary}\label{domainconvergence}
Let $\{f_n\}_{n \in \zz}$ be a sequence of slice regular functions in $\mc{SR}(\OO)$. For each circular open subset $\OO'$ of $\OO$, the following assertions are equivalent.
\begin{enumerate}
\item $\OO'$ is the domain of convergence of $f=\sum_{n \in \zz} f_n$.
\item There exists $J \in \s_A$ such that:
\begin{itemize}
\item $f$ converges totally on each compact subset of $\OO'_J$; 
\item for every open subset $U_J$ of $\OO_J$ preserved by $z \mapsto z^c$ and with $\OO'_J\subsetneq U_J$, there exists a point $z \in U_J\setminus\OO_J'$ such that $\sum_{n \in \nn} f_n(z)$ or $\sum_{m \in \nn} f_{-m}(z)$ does not converge.
\end{itemize}
\item There exist $J,K \in \s_A$ with $J-K$ invertible such that:
\begin{itemize}
\item $f$ converges totally on each compact subset of $\OO'\cap(\mathscr{H}_J \cup\mathscr{H}_K)$;
\item for every circular open subset $U$ of $\OO$ with $\OO'\subsetneq U$ there exists a point $z \in (U\setminus\OO')\cap(\mathscr{H}_J \cup\mathscr{H}_K)$ such that $\sum_{n \in \nn} f_n(z)$ or $\sum_{m \in \nn} f_{-m}(z)$ does not converge.
\end{itemize}
\end{enumerate}
\end{corollary}

We point out that, both in the previous theorem and in the subsequent corollary, case number {\it 3} is included for the sake of completeness. It will not be used in the rest of the paper.


\section{Laurent series and expansions} \label{sec:power-laurent}

The present section introduces a notion of Laurent series that generalizes the concept of regular power series introduced in~\cite{powerseries,expansionsalgebras} and that of quaternionic Laurent series given in~\cite{singularities}. 

\begin{definition}
Let $y \in Q_A$. For any sequence $\{a_n\}_{n\in \zz}$ in $A$, the series
\begin{equation}
\sum_{n \in \zz}(x-y)^{\punto n} \cdot a_n
\end{equation}
is called the \emph{Laurent series centered at $y$ associated with $\{a_n\}_{n\in \zz}$}. If $a_n = 0$ for all $n<0$, then it is called the \emph{power series centered at $y$ associated with $\{a_n\}_{n\in \nn}$}.
\end{definition}

We note that $(x-y)^{\punto n} \cdot a_n$ may be different from the function $(x-y)^{\punto n} a_n$:
\begin{example}\label{ex:productbyconstant}
If $A=\oo = \hh+ l \hh$ with its standard basis $1,i,j,k,l,li,lj,lk$, then 
\[f(x) := (x+i)^{\punto 2} \cdot l = (x^2+2xi -1)\cdot l = x^2l+2x(il) -l\]
is a different function than $g(x):=(x+i)^{\punto 2} l = (x^2+2xi -1)l = x^2l+2(xi)l -l$.
\end{example}

For this reason, the next lemma and the subsequent definition will be useful.

\begin{lemma}\label{productpreservedslice}
Let $f,g: \OO \to A$ be slice functions and let $J \in \s_A$. If $f(\OO_J)\subseteq\cc_J$, then 
\[(f\cdot g)(z) = f(z)g(z)\]
for all $z \in \OO_J$. Moreover, if $\{1,J, J_1, J J_1,\ldots, J_h, JJ_h\}$ is a splitting basis of $A$ associated with $J$, if $J_0:=1$ and if $\{g_k:\OO_J \lra \cc_J\}_{k=0}^h$ are such that $g_{|_{\OO_J}}=\sum_{k=0}^h g_k J_k$, then
\[(f\cdot g)(z) = \sum_{k=0}^h (f(z) g_k(z)) J_k\]
for all $z \in \OO_J$.
\end{lemma}

\begin{proof}
We begin by establishing the first equality. According to~\cite[Theorem 3.4]{gpsalgebra},
\[(f\cdot g)(x) = f(x) \vs  g(x)+\im(x) \big(f(x)  g'_s(x)\big) - (\im(x),  f'_s(x),g(x^c)).\]
Under the hypothesis $f(\OO_J)\subseteq \cc_J$, for all $z \in \cc_J$ we have that $\im(z),f'_s(z) \in \cc_J$, whence $(\im(z), f'_s(z),g(z^c))=0$ by Artin's theorem~\cite[Theorem 3.1]{schafer}. Moreover, since $f(z) \in \cc_J$,
\[\im(z) \big(f(z)  g'_s(z)\big) = \big(\im(z) f(z) \big) g'_s(z) = \big(f(z) \im(z)\big) g'_s(z) =  f(z) \big(\im(z) g'_s(z) \big).\]
Thus,
\[(f\cdot g)(z) = f(z) \vs g(z)+f(z) \big(\im(z) g'_s(z) \big) = f(z) (\vs g(z) + \im(z) g'_s(z)) = f(z) g(z),\]
which is precisely the first equality.

As for the second one, if $g(z)=\sum_{k=0}^h g_k(z) J_k$ in $\OO_J$, then
\[(f\cdot g)(z) = f(z) g(z) = \sum_{k=0}^h f(z) (g_k(z) J_k) = \sum_{k=0}^h (f(z) g_k(z)) J_k\]
where $(f(z),g_k(z),J_k)=0$ by a further application of Artin's theorem.
\end{proof}

\begin{example}
Over $\oo$, the functions $(x+i)^{\punto 2}\cdot l$ and $(x+i)^{\punto 2} l$ do not coincide identically, although they coincide in $\cc_i$. By the Representation Formula~\eqref{rep2}, this implies that $(x+i)^{\punto 2} l$ is not a slice function (whence not a slice regular function).
\end{example}

\begin{definition}
Let $y \in \cc_J\subseteq Q_A$. For all $R,R_1,R_2 \in [0, +\infty]$ with $R_1<R_2$, we set 
\[B_J (y,R):=\{z \in \cc_J \, | \, |z-y|<R\},\]
\[A_J (y,R_1,R_2):=\{z \in \cc_J \, | \, R_1<|z-y|<R_2\}.\]
Furthermore, we let $\OO(y,R)$ denote the circular set $\OO\subseteq Q_A$ such that 
\[\OO_J = B_J (y,R) \cap B_J (y^c,R);\] 
the latter being the largest subset of $B_J (y,R)$ preserved by $z \mapsto z^c$. Similarly, we let $\OO(y,R_1,R_2)$ denote the circular subset of $Q_A$ whose $J$-slice is $A_J (y,R_1,R_2) \cap A_J (y^c,R_1,R_2)$.
\end{definition}

We are now ready to study the convergence of Laurent series.

\begin{theorem}\label{thm:Abel-power}
Let $y \in \cc_J\subseteq Q_A$ and let $\{a_n\}_{n \in \zz}$ be a sequence in $A$. Set
\begin{center}
$R_1 = \limsup_{m \to +\infty} \|a_{-m}\|_A^{1/m} \;$ and $\; 1/R_2 = \limsup_{n \to +\infty} \|a_n\|_A^{1/n}$.
\end{center}
If $R_1<R_2$, then the Laurent series 
\begin{equation}\label{eq:Abel-power}
\mr{P}(x) = \sum_{n \in \zz} (x-y)^{\punto n}\cdot a_n
\end{equation}
converges totally on every compact subset of $A_J(y,R_1,R_2)$ and it does not converge at any point $x \in \cc_J \setminus \overline{A_J(y,R_1,R_2)}$. If there exists $n<0$ such that $a_n \neq 0$, then $\OO(y,R_1,R_2)$ is its domain of convergence (with $R_1=0$ if there are finitely many such $n$'s). If, on the other hand, $\mr{P}(x)$ is a power series, then $\OO(y,R_2)$ is its domain of convergence. Finally, if the domain of convergence $\OO$ of $\mr{P}(x)$ is not empty, then the sum is a slice regular function $\mr{P} \in \sr(\OO)$.
\end{theorem}

\begin{proof}
Let $f_n(x) := (x-y)^{\punto n}\cdot a_n$ on $Q_A\setminus\s_y$. By Lemma~\ref{productpreservedslice}, $f_n(z) = (z-y)^n a_n$ for all $z \in \cc_J \setminus\{y,y^c\}$. Remark~\ref{multiplicativeestimates} implies that
\[\|(z-y)^n a_n\|_A \leq C_A |z-y|^n \|a_n\|_A\,,\]
\[\|(z-y)^n a_n\|_A \geq c_A |z-y|^n \|a_n\|_A\,.\]
By the root criterion, $\mr{P}(z)=\sum_{n \in \zz} f_n(z)$ converges totally on the compact subsets of $A_J (y,R_1,R_2)$ and it does not converge at any point of $\cc_J \setminus \overline{A_J(y,R_1,R_2)}$.

In case $\mr{P}(z)$ is a power series rather than a genuine Laurent series, the same technique proves that it converges totally on the compact subsets of $B_J(y,R_2)$ and it does not converge at any point of $\cc_J \setminus \overline{B_J(y,R_2)}$.

By Corollary~\ref{domainconvergence}, the domain of convergence $\OO$ of $\mr{P}(x)$ is either $\OO(y,R_1,R_2)$ or $\OO(y,R_2)$ depending on whether $\mr{P}(x)$ is a genuine Laurent series or a power series.
Finally, if $\OO \neq \emptyset$, then the sum is a slice regular function $P \in \sr(\OO)$ by Remark~\ref{regularsum}.
\end{proof}

We adopted here a new approach, which differs significantly from the techniques of~\cite{powerseries,expansionsalgebras,singularities}. In those papers, convergence was studied by estimating $\|(x-y)^{\punto n}\|_A$ in terms of a distance $\sigma_A$ and a pseudodistance $\tau_A$. Let us recall the definition of these functions.

\begin{definition}
We define the functions $\sigma_A,\tau_A:Q_A \times Q_A \lra \rr$ by setting
\begin{eqnarray}
&\sigma_A(x,y):=& \left\{
\begin{array}{ll}
|x-y|  \mathrm{\ if\ } x,y \mathrm{\ lie\ in\ the\ same\ } \cc_J \vspace{.3em}\\
\sqrt{\left(\re(x)-\re(y)\right)^2 + \left(|\im(x)| + |\im(y)|\right)^2}  \mathrm{\ otherwise}
\end{array}
\right. \ ,\\
&& \nonumber\\
&\tau_A(x,y):=& \left\{
\begin{array}{l}
|x-y|  \mathrm{\ if\ } x,y \mathrm{\ lie\ in\ the\ same\ } \cc_J \vspace{.3em}\\
\sqrt{\left(\re(x)-\re(y)\right)^2 + \left(|\im(x)| - |\im(y)|\right)^2}   \mathrm{\ otherwise}
\end{array}
\right.\ .
\end{eqnarray}
\end{definition}

If we set
\begin{align*}
\Sigma(y,R) &:= \{x \in Q_A \, | \, \sigma_A(x,y)<R\},\\
\Sigma(y,R_1,R_2)&:=\{x \in Q_A \, | \, \tau_A(x,y)>R_1, \sigma_A(x,y) <R_2\}.
\end{align*}
then the next remark proves that our new result is consistent with those previous works.

\begin{remark}
The equalities
\begin{align*}
\Sigma(y,R) &= \OO(y,R) \cup B_J(y,R),\\
\Sigma(y,R_1,R_2) &= \OO(y,R_1,R_2) \cup A_J(y,R_1,R_2)
\end{align*}
hold.
\end{remark}

This allows to call the quantities $R_1$ and $R_2$ appearing in Theorem~\ref{thm:Abel-power} the \emph{inner} and \emph{outer radius of convergence} of the Laurent series \eqref{eq:Abel-power}, respectively.
On the other hand, our new approach spared us many lengthy computations and a detailed study of the topology of $\Sigma(y,R_1,R_2)$. Indeed, $\sigma_A$ and $\tau_A$ are only lower (resp., upper) semicontinuous in the Euclidean topology of $A$.

We now come to Laurent series expansions for slice regular functions.

\begin{theorem}\label{Laurent}
Let $f: \OO \lra A$ be a slice regular function. Suppose that $y \in Q_A$ and $R_1,R_2 \in [0,+\infty]$ are such that $R_1<R_2$ and $\Sigma(y,R_1,R_2) \subseteq \OO$. Then there exists a (unique) sequence $\{a_n\}_{n \in \zz}$ in $A$ such that
\begin{equation}\label{laurentformula}
f(x) = \sum_{n \in \zz} (x-y)^{\punto n}\cdot a_n
\end{equation}
in $\Sigma(y,R_1,R_2)$. If, moreover, $\Sigma(y,R_2)\subseteq \OO$, then for all $n<0$ we have $a_n = 0$ and formula~\eqref{laurentformula} holds in $\Sigma(y,R_2)$.
\end{theorem}

\begin{proof}
Since $\OO \supseteq \Sigma(y,R_1,R_2)$, we have that
\[
\OO_J \supseteq A_J(y,R_1,R_2).
\]
Choose $J_1,\ldots,J_h \in \s_A$ in such a way that $\{1,J, J_1, J J_1,\ldots, J_h, JJ_h\}$ is a splitting basis of $A$ associated with $J$ and let $\{f_k:\OO_J \lra \cc_J\}_{k=0}^h$ be the holomorphic functions such that $f_{|_{\OO_J}}=\sum_{k=0}^h f_k J_k$, where $J_0$ stands for $1$. If
\begin{equation*}
f_k(z) = \sum_{n \in \zz} (z-y)^n a_{n,k}
\end{equation*}
is the Laurent series expansions of $f_k$ in $A_J(y,R_1,R_2)$, where $a_{n,k} \in \cc_J$ for all $n \in \zz$, then for all $z \in A_J(y,R_1,R_2)$
\begin{equation*}
f(z)=\sum_{k=0}^h f_k(z) J_k=\sum_{k=0}^h \sum_{n \in \zz} \left((z-y)^n a_{n,k}\right) J_k=\sum_{n \in \zz} \sum_{k=0}^h(z-y)^n  \left(a_{n,k} J_k\right)=\sum_{n \in \zz} (z-y)^n a_n
\end{equation*}
where $a_n := \sum_{k=0}^h a_{n,k} J_k$. Since $\sum_{n \in \zz} (z-y)^n a_n$ converges in $A_J(y,R_1,R_2)$, by Theorem~\ref{thm:Abel-power} we must have $\limsup_{m \to +\infty} \|a_{-m}\|_A^{1/m} \leq R_1$ and $\limsup_{n \to +\infty} \|a_n\|_A^{1/n} \leq1/R_2$. As a consequence, the domain of convergence of the series in equation \eqref{laurentformula} includes $\OO(y,R_1,R_2)$. 

If $\OO(y,R_1,R_2) \neq \emptyset$, then the series in equation \eqref{laurentformula} defines a slice regular function $g: \OO(y,R_1,R_2) \lra A.$ Since $g$ coincides with $f$ in $\OO(y,R_1,R_2) \cap \cc_J$, by Formula \eqref{rep2}, it must coincide with $f$ in $\OO(y,R_1,R_2)$.

Hence equality \eqref{laurentformula} holds throughout $\Sigma(y,R_1,R_2) = A_J(y,R_1,R_2) \cup \OO(y,R_1,R_2)$, as desired.

Finally, if $\Sigma(y,R_2)\subseteq \OO$, then the slice regularity of $f$ near $y$ implies that every component $f_k$ is holomorphic near $y$ and the same steps lead to an expansion of the form $f(x) = \sum_{n \in \nn} (x-y)^{\punto n}\cdot a_n$ (with non negative exponents) valid in $\Sigma(y,R_2)$.
\end{proof}

\begin{example}\label{ex:expansion}
Let $J \in \s_A$ and let $f:Q_A \setminus \s_A \to A$ be defined by
\begin{equation*}
f(x) = (x+J)^{-\punto } = (x^2+1)^{-1}(x-J).
\end{equation*}
For all $z \in \cc_J \setminus \{\pm J\}$, we have $f(z) = (z+J)^{-1}$. Therefore, 
\[f(z) = \sum_{n \in \nn} (-1)^n(z-J)^n(2J)^{-n-1}\]
in $A_J(J,0,2)$ and
\[f(x) = \sum_{n \in \nn} (-1)^n(x-J)^{\punto n}(2J)^{-n-1}\]
in $\Sigma(J,0,2)$.
\end{example}

As in the complex case, the coefficients in expansion~\eqref{laurentformula} can be computed by means of line integral formulae. Let us first clarify the notation used for line integrals in the setting of our alternative algebra $A$. 

\begin{definition}\label{lineintegral}
Let $J \in \s_A$, let $\gamma:[0,1] \lra \cc_J$ be a piecewise-$\mathscr{C}^1$ curve and let $f:|\gamma| \lra A$ be a continuous function defined on the support $|\gamma|=\gamma([0,1])$ of $\gamma$. Fix a splitting basis $\{1,J,J_1,JJ_1,\ldots,J_h,JJ_h\}$ of $A$ associated with $J$ and denote by $\{f_k:|\gamma| \lra \cc_J\}_{k=0}^h$ the functions such that $f=\sum_{k=0}^hf_kJ_k$ on $|\gamma|$, where $J_0:=1$. Then we define
\[
\int_{\gamma}d\zeta f(\zeta):=\sum_{k=0}^h \left( \int_{\gamma} f_k(\zeta) d\zeta \right) J_k.
\]
\end{definition}
The definition is well-posed, i.e., for $J$ fixed it does not depend on the choice of the associated splitting basis of $A$. 

\begin{proposition}
In the hypotheses of Theorem~\ref{Laurent}, let us consider any circle $\gamma:[0,1] \lra \cc_J$ defined by $s \mapsto y+ re^{2\pi Js}$ for some $r \in (R_1,R_2)$. Then
\begin{eqnarray} \label{eq:a_n}
a_n&=&(2 \pi J)^{-1}\int_{\gamma} d\zeta \, (\zeta-y)^{-n-1} f(\zeta)\\
&=& r^{-n} \int_0^1 e^{-2\pi nJ s}\, f(y+ re^{2\pi Js})\,ds \notag
\end{eqnarray}
for all $n \in \zz$.
\end{proposition}

\begin{proof}
Let $\{1,J,J_1,JJ_1,\ldots,J_h,JJ_h\}$ be a splitting basis of $A$ associated with $J$ and let $J_0:=1$. By direct inspection in the proof of Theorem~\ref{Laurent}, $a_n = \sum_{k=0}^h a_{n,k} J_k$ where $\{a_{n,k}\}_{n \in \zz} \subseteq \cc_J$ is the sequence of the coefficients of the Laurent expansion of the holomorphic function $f_k:\OO_J \to \cc_J$ at $y$. The equality
\[
\frac1{2 \pi J}\int_{\gamma} \frac{f_k(\zeta)}{(\zeta-y)^{n+1}}d\zeta=a_{n,k}\,,
\]
is a well-known fact in the theory of holomorphic functions of one complex variable. Now let $g(x):= (x-y)^{-\punto(n+1)}\cdot f(x)$. By Lemma~\ref{productpreservedslice},
\[
g(\zeta) = \sum_{k=0}^h ((\zeta-y)^{-n-1} f_k(\zeta))J_k
\]
for all $\zeta \in \cc_J\setminus\{y,y^c\}$. According to Definition~\ref{lineintegral},
\[\frac1{2 \pi J}\int_{\gamma} d\zeta g(\zeta)=  \sum_{k=0}^h \left(\frac1{2 \pi J}\int_{\gamma} \frac{f_k(\zeta)}{(\zeta-y)^{n+1}}d\zeta\right)J_k = \sum_{k=0}^h a_{n,k} J_k = a_n.\] 
Finally, $g(\zeta)=(\zeta-y)^{-n-1} f(\zeta)$ by a further application of Lemma~\ref{productpreservedslice}.
\end{proof}

\begin{corollary}
Suppose that the hypotheses of Theorem~\ref{Laurent} hold. For each $J \in \s_A$, we have $f(A_J(y,R_1,R_2))\subseteq\cc_J$ if, and only if, $a_n \in \cc_J$ for all $n \in \zz$. Moreover, if $y \in \rr$ and if $f$ is slice preserving in $\Sigma(y,R_1,R_2)$, then $a_n \in \rr$ for all $n \in \zz$.
\end{corollary}

\begin{corollary}\label{coefficientsareslicefunctions}
Suppose that the hypotheses of Theorem~\ref{Laurent} are fulfilled at a point $y \in Q_A \setminus \rr$, so that, for all $w \in \s_y$, the function $f$ expands at $w$ into a Laurent series with coefficients $a_n(w)$. Then for all $n \in \zz$, the map $w \mapsto a_n(w)$ is a slice function on $\s_y$; equivalently, $a_n : \s_y \to A$ is left affine over $A$.
\end{corollary}

\begin{proof}
Let us fix $n \in \zz$ and suppose that $\s_y = \alpha +\beta \s_A$ with $\alpha, \beta \in \rr$ and $\beta>0$. According to Formula~\eqref{eq:a_n},
\[a_n(\alpha+\beta J) = r^{-n} \int_0^1 e^{-2\pi nJ s}\, f(\alpha+\beta J+ re^{2\pi Js})\,ds
\]
for all $r \in (R_1,R_2)$ (see the hypotheses of Theorem~\ref{Laurent}). Now, for $s \in [0,1]$ fixed, the functions 
\[\varphi_s(\alpha+\beta J):=e^{-2\pi nJ s} = \cos(2\pi ns)-J \sin(2\pi ns)\]
and $\psi_s(\alpha+\beta J):= f(\alpha+\beta J+ re^{2\pi Js})$ are slice functions on $\s_y$ and so is their slice product $\varphi_s \cdot \psi_s$. Since $\varphi_s$ is slice preserving, we have $(\varphi_s \cdot \psi_s)(\alpha+\beta J) = \varphi_s(\alpha+\beta J) \psi_s(\alpha+\beta J)$. It follows immediately that
\[a_n(\alpha+\beta J) = r^{-n} \int_0^1 \varphi_s(\alpha+\beta J) \psi_s(\alpha+\beta J)\,ds\,. 
\]
is a slice function on $\s_y$.
\end{proof}


\section{Types of singularities}\label{sec:types}

In this section, we undertake a first classification of the singularities of slice regular functions.
Following the approach used for the quaternionic case in~\cite{singularities}, we give the next definition.

\begin{definition}\label{def:Laurent-order}
Let $f$ be a slice regular function on a circular open subset $\OO$ of $Q_A$. We say that a point $y \in Q_A$ is a \emph{singularity} for $f$ if there exists $R>0$ such that $\Sigma(y,0,R) \subseteq \OO$, so that $f$ admits a Laurent expansion
\begin{equation}\label{eq:classification}
f(x) = \sum_{n \in \zz} (x-y)^{\punto n}\cdot a_n
\end{equation}
with an inner radius of convergence equal to $0$ and a positive outer radius of convergence. This expansion will be called \emph{the Laurent expansion of $f$ at $y$}.

The point $y$ is said to be a \emph{pole} for $f$ if there exists an $m\geq0$ such that $a_{-k} = 0$ for all $k>m$; the minimum such $m$ is called the \emph{order} of the pole and denoted as $\ord_f(y)$. If $y$ is not a pole, then it is called an \emph{essential singularity} for $f$ and $\ord_f(y) := +\infty$. Finally, $y$ is called a \emph{removable singularity} if $f$ extends to a slice regular function in $\mc{SR}(\widetilde \OO)${, where $\widetilde \OO$ is a circular open subset of $Q_A$ containing $y$.}
\end{definition}

We point out that the set of convergence $\Sigma(y,0,R) = A_J(y,0,R) \cup \OO(y,0,R)$ of the Laurent expansion \eqref{eq:classification} used in the classification does not always contain a set of the form $U \setminus \s_y$ for some circular neighborhood $U$ of $\s_y$ in $Q_A$. For this reason, a pole of order $0$ is not necessarily a removable singularity for $f$:

\begin{example}
Let $J \in \s_A$ and let
\begin{equation*}
f(x) = (x+J)^{-\punto } = (x^2+1)^{-1}(x-J).
\end{equation*}
The point $-J$ is clearly a pole of order $1$ for $f$. The expansion at $J$ performed in Example~\ref{ex:expansion} proves that $J$ is a pole of order $0$ . On the other hand, $J$ is not a removable singularity for $f$. Indeed, for each circular neighborhood $U$ of $J$ in $Q_A$, the modulus $|f|$ is unbounded in $U \setminus \s_A$ because $\s_A$ includes $-J$. 
\end{example}

For the same reason, in~\cite{singularities}, the study of $f$ near each pole $y$ was addressed by expressing $f$ as a quotient of functions, both slice regular in a neighborhood of $y$. In the present paper, we do not simply generalize this construction to more general algebras. We instead provide a more refined tool, which is new even for the quaternionic case: a Laurent-type expansion valid in circular open subsets of $Q_A$. The next two sections will be devoted to this matter.


\section{Spherical Laurent series} \label{sec:spherical-laurent}

In this section we will generalize, in the Laurent sense, the notion of spherical series introduced in~\cite{expansion} for quaternions and in \cite{expansionsalgebras} for alternative *-algebras.

\begin{definition}
Let $y \in Q_A$. For every $k \in \zz$, we set 
\begin{align*}
\stx_{y,2k}(x) &:= \Delta_y^{\punto k}(x) =  \Delta_y(x)^k,\\
\stx_{y,2k+1}(x) &:= \Delta_y^{\punto k}(x)\cdot (x-y)=\Delta_y(x)^k (x-y).
\end{align*}
For any sequence $\{c_n\}_{n\in \zz}$ in $A$, the series
\begin{equation}\label{eq:sphericallaurent}
\sum_{n \in \zz} \stx_{y,n} \cdot c_n
\end{equation}
is called the \emph{spherical Laurent series centered at $y$ associated with $\{c_n\}_{n\in \zz}$}. If $c_n = 0$ for every $n<0$, then it is called the \emph{spherical series centered at $y$ associated with $\{c_n\}_{n \in \nn}$}.
\end{definition}

We point out that, for $n=2k+1$,
\[\stx_{y,n}(x) \cdot c_n = \Delta_y^{\punto k}(x)\cdot (x-y)\cdot c_n = \Delta_y^{\punto k}(x)((x-y)\cdot c_n) = \Delta_y^{k}(x)((x-y) c_n)\]
(where we have used the fact that $\Delta_y$ is a slice preserving function, whence an element of the associative nucleus of $\mc{SR}(Q_A)$) may well differ from
\[\stx_{y,n}(x) c_n = (\Delta_y^{k}(x)(x-y))c_n\]
in a nonassociative setting.
In order to study the convergence of~\eqref{eq:sphericallaurent}, the following notions introduced in \cite{expansionsalgebras,expansion} will be useful.

\begin{proposition}\label{pseudometric}
The function
\begin{equation} \label{eq:sto}
\sto(x,y):=\sqrt{\|\Delta_y(x)\|_A}
\end{equation}
is a pseudodistance on $Q_A$, called the \emph{Cassini pseudodistance}, and it is continuous with respect to the relative Euclidean topology of $Q_A$. For all $J \in \s$ and for all $y,z \in \cc_J$ the equality
\[\sto(y,z):=\sqrt{|z-y|\,|z-y^c|}\]
holds; for all $y \in \rr, x \in Q_A$ we have $\sto(x,y) = \|x-y\|_A$. The \emph{Cassini ball} 
\[\Sto(y,r):=\{x \in Q_A \, | \, \sto(x,y)<r\}\]
is a circular open neighborhood of $\s_y$ in $Q_A$ and the \emph{Cassini shell} 
\[\Sto(y,r_1,r_2):=\{x \in Q_A \, | \, r_1<\sto(x,y)<r_2\}.\]
is a circular open subset of $Q_A$. In the special case when $y \in \rr$, they coincide with $Q_A \cap B(y,r)$ and with $\{x \in Q_A\,|\, r_1<\|x-y\|_A<r_2\}$, respectively.
\end{proposition}

The next lemma will also be useful.

\begin{lemma}\label{technical}
Fix $J \in \s_A$ and $y = \alpha + \beta J \in \cc_J$. Then
\[\sqrt{\sto(y,z)^2+\beta^2}-\beta \leq |z-y| \leq \sqrt{\sto(y,z)^2+\beta^2}+\beta\]
for all $z \in \cc_J$.
\end{lemma}

\begin{proof}
By Proposition~\ref{pseudometric}, for all $z \in \cc_J$
\[\sto(y,z)^2 = |z-y|\,|z-y^c|.\]
If $|z-y|<\sqrt{\sto(y,z)^2+\beta^2}-\beta$, then automatically
\[|z-y^c| \leq |z-y| + |y-y^c| < \sqrt{\sto(y,z)^2+\beta^2}-\beta + 2 \beta = \sqrt{\sto(y,z)^2+\beta^2}+\beta,\]
whence the contradiction
\[\sto(y,z)^2 = |z-y|\,|z-y^c| < \sto(y,z)^2+\beta^2 - \beta^2 = \sto(y,z)^2.\]
An inequality $|z-y|>\sqrt{\sto(y,z)^2+\beta^2}+\beta$ would lead to an analogous contradiction.
\end{proof}

\begin{theorem} \label{thm:Abel-spherical}
Let $y \in Q_A$, let $\{c_n\}_{n \in \zz}$ be a sequence in $A$, set
\begin{center}
$r_1 := \limsup_{m \to +\infty} \|c_{-m}\|_A^{1/m} \;$ and $\; 1/r_2 := \limsup_{n \to +\infty} \|c_n\|_A^{1/n}$,
\end{center}
and consider the spherical Laurent series
\begin{equation}
\mr{S}(x)=\sum_{n \in \zz} \stx_{y,n}(x) \cdot c_n.
\end{equation}
If there exists $n<0$ such that $c_n \neq 0$, then the domain of convergence is the Cassini shell $\Sto(y,r_1,r_2)$ (with $r_1=0$ if there are finitely many such $n$'s). If, to the contrary, $c_n=0$ for all $n<0$ then $\mr{S}(x)$ is a spherical series and its domain of convergence is the Cassini ball $\Sto(y,r_2)$. If the domain of convergence $\OO$ is not empty then the sum is slice regular in $\OO$.
\end{theorem}

\begin{proof}
If we let $f_n(x):= \stx_{y,n}(x) \cdot c_n$, then, by Lemma~\ref{productpreservedslice}, we have $f_n(z) =  \stx_{y,n}(z) c_n$ for all $z \in \cc_J$. Hence, by Remark~\ref{multiplicativeestimates},
\[c_A \|\stx_{y,n}(z)\|_A\,\|c_n\|_A \leq \|f_n(z)\|_A \leq C_A \|\stx_{y,n}(z)\|_A\, \|c_n\|_A\,.\]
Moreover, $\|\stx_{y,2k}(z)\|_A = |\Delta_y(z)|^k = \sto(y,z)^{2k}$ while $\|\stx_{y,2k+1}(z)\|_A = |\Delta_y^k(z)(z-y)| = \sto(y,z)^{2k} |z-y|$ implies, by Lemma~\ref{technical}, that
\[ \sto(y,z)^{2k} \left(\sqrt{\sto(y,z)^2+\beta^2}-\beta\right)\leq \|\stx_{y,2k+1}(z)\|_A\leq\sto(y,z)^{2k} \left(\sqrt{\sto(y,z)^2+\beta^2}+\beta\right)\]
where $\beta:=|\im(y)|$.

The root criterion allows us to conclude that $\mr{S}(z) = \sum_{n \in \zz}f_n(z)$ converges totally on the compact subsets of the Cassini annulus $\Sto_J(y,r_1,r_2):=\Sto(y,r_1,r_2) \cap \cc_J$ and that it does not converge at any point of $\cc_J \setminus \overline{\Sto_J(y,r_1,r_2)}$.

By Corollary~\ref{domainconvergence}, the domain of convergence $\OO$ of $\mr{S}(x)$ is either $\Sto(y,r_1,r_2)$ or $\Sto(y,r_2)$ depending on whether $\mr{S}(x)$ is a genuine spherical Laurent series or simply a spherical series.
Finally, if $\OO \neq \emptyset$, then the sum is a slice regular function $\mr{S} \in \sr(\OO)$ by Remark~\ref{regularsum}.
\end{proof}

\begin{remark}\label{rearrange}
In the hypotheses of Theorem~\ref{thm:Abel-spherical}, 
\[\stx_{y,2k}(x) \cdot c_{2k} + \stx_{y,2k+1}(x) \cdot c_{2k+1} = \Delta_y^{k}(x) (c_{2k} + (x-y)c_{2k+1}),\]
whence the two following consequences.
\begin{enumerate}
\item If $y' \in \s_y$ and if we set
\begin{align*}
&c'_{2k}:=c_{2k}+(y'-y)c_{2k+1},\\
&c'_{2k+1}:=c_{2k+1},
\end{align*}
then the series $\sum_{n \in \zz}\stx_{y',n}(x) \cdot c'_n$ has the same domain of convergence $\OO$ as $\mr{S}(x)$ and
\[\mr{S}(x) = \sum_{n \in \zz}\stx_{y',n}(x) \cdot c'_n\]
in $\OO$.
\item If the domain of convergence $\OO$ of $\mr{S}(x)$ is not empty and if we set
\begin{align}
&u_k:=c_{2k+1}, \label{eq:uk}\\
&v_k:=c_{2k}-y c_{2k+1}, \label{eq:vk}
\end{align}
then the domain of convergence of the series $\sum_{k \in \zz}\Delta_y^k(x)(xu_k+v_k)$ includes $\OO$ and
\[\mr{S}(x)=\sum_{k \in \zz}\Delta_y^k(x)(xu_k+v_k)\]
in $\OO$.
\end{enumerate}
\end{remark}


\section{Spherical Laurent expansions} \label{sec:spherical-laurent-expansions}

We now proceed to expand slice regular functions into spherical Laurent series. We begin with the special case $A=\cc$.

\begin{lemma} \label{lem:C}
Suppose $A=\cc$ and let $\OO\subseteq \cc$ be an open set preserved by complex conjugation and take any $g \in \mc{SR}(\OO)$; that is, any holomorphic function $g: \OO \to \cc$. Let $w \in \cc$ and $r_1,r_2 \in [0,+\infty]$ be such that $\Sto(w,r_1,r_2) \subseteq \OO$. If $r_1<r_2$, then we can define for all $n \in \zz$ and for $r$ in the interval $(r_1,r_2)$
\begin{align}\label{complexsphericalnumber}
c_n:=\frac{1}{2\pi i}\int_{\partial \Sto(w,r)}\frac{g(\zeta)}{\stx_{w,n+1}(\zeta)} \, d\zeta,
\end{align}
which does not depend on the choice of $r$. The spherical Laurent series centered at $w$ associated with $\{c_n\}_{n \in \zz}$ has a domain of convergence that includes the Cassini shell $\Sto(w,r_1,r_2)$ and
\begin{align*}
g(z) = \sum_{n \in \zz}\stx_{w,n}(z) \,c_n
\end{align*}
for all $z \in \Sto(w,r_1,r_2)$. If, moreover, $\Sto(w,r_2)\subseteq \OO$, then $c_n=0$ for all $n<0$ and the expansion is valid in the whole Cassini ball $\Sto(w,r_2)$.
\end{lemma}

\begin{proof}
Let us first establish that the definition of $c_n$ is independent of the choice of $r \in (r_1,r_2)$. If $r_1<r<r'<r_2$, then $\Sto(w,r_1,r_2)\supset \partial \Sto(w,r,r')$ and we can apply the complex Cauchy theorem to the holomorphic function $g$ to get the desired equality
\[0 = \int_{\partial \Sto(w,r,r')}\frac{g(\zeta)}{\stx_{w,n+1}(\zeta)} \, d\zeta = \int_{\partial \Sto(w,r')}\frac{g(\zeta)}{\stx_{w,n+1}(\zeta)} \, d\zeta-\int_{\partial \Sto(w,r)}\frac{g(\zeta)}{\stx_{w,n+1}(\zeta)} \, d\zeta\,.\]

Now let $z \in \Sto(w,r_1,r_2)$ and choose $R_1,R_2 \in \rr$ such that $r_1<R_1<R_2<r_2$ and such that $z \in \Sto(w,R_1,R_2)$. By the complex Cauchy integral formula,
\begin{align*}
g(z) = \frac{1}{2\pi i}\int_{\partial \Sto(w,R_2)}\frac{g(\zeta)}{\zeta-z}\, d\zeta - \frac{1}{2\pi i}\int_{\partial \Sto(w,R_1)}\frac{g(\zeta)}{\zeta-z}\, d\zeta.
\end{align*}
By direct computation (see~\cite[Equation (5.9)]{expansionsalgebras}), for all $n \in \nn$
\begin{equation} \label{eq:cauchy-kernel}
\frac{1}{\zeta-z}=\sum_{k=0}^n\frac{\stx_{w,k}(z)}{\stx_{w,k+1}(\zeta)}+\frac{\stx_{w,n+1}(z)}{\stx_{w,n+1}(\zeta)}\, \frac{1}{\zeta-z}.
\end{equation}
Applying the same equality at $\bar w$ and switching the roles of $\zeta$ and $z$, we have
\begin{align*}
-\frac{1}{\zeta-z}&=\sum_{k=0}^n\frac{\stx_{\bar w,k}(\zeta)}{\stx_{\bar w,k+1}(z)}+\frac{\stx_{\bar w,n+1}(\zeta)}{\stx_{\bar w,n+1}(z)} \, \frac{1}{z-\zeta}\\
&= \sum_{k=0}^n \frac{\stx_{w,-k-1}(z)}{\stx_{w,-k}(\zeta)}+\frac{\stx_{\bar w,n+1}(\zeta)}{\stx_{\bar w,n+1}(z)} \, \frac{1}{z-\zeta}\,,
\end{align*}
where we have taken into account that $\stx_{\bar w,m} = \frac{1}{\stx_{w,-m}}$, whence
\[-\frac{1}{\zeta-z}= \sum_{h=-m}^{-1} \frac{\stx_{w,h}(z)}{\stx_{w,h+1}(\zeta)}+\frac{\stx_{\bar w,m}(\zeta)}{\stx_{\bar w,m}(z)}\, \frac{1}{z-\zeta}\]
for all $m\geq 1$. Thus, for all $m,n \in \nn$ with $m \geq 1$,
\begin{align*}
g(z) &= \sum_{k=0}^n\frac{\stx_{w,k}(z)}{2\pi i}\int_{\partial \Sto(w,R_2)}\frac{g(\zeta)}{\stx_{w,k+1}(\zeta)}\, d\zeta + \frac{\stx_{w,n+1}(z)}{2\pi i}\int_{\partial \Sto(w,R_2)}\frac{g(\zeta)}{\stx_{w,n+1}(\zeta)}\, \frac{d\zeta}{\zeta-z}\\
&+\sum_{h=-m}^{-1} \frac{\stx_{w,h}(z)}{2\pi i}\int_{\partial \Sto(w,R_1)}\frac{g(\zeta)}{\stx_{w,h+1}(\zeta)}\, d\zeta + \frac{1}{2\pi i}\frac{1}{\stx_{\bar w,m}(z)}\int_{\partial \Sto(w,R_1)}{g(\zeta)}\, \frac{\stx_{\bar w,m}(\zeta)}{z-\zeta}\, d\zeta\,.
\end{align*}
If we set 
\[\mr{F}_n(z):= \frac{\stx_{w,n+1}(z)}{2\pi i}\int_{\partial \Sto(w,R_2)} \frac{g(\zeta)}{\stx_{w,n+1}(\zeta)}\, \frac{d\zeta}{\zeta-z},\]
\[\mr{F}_{-m}(z):= \frac{1}{2\pi i} \frac{1}{\stx_{\bar w,m}(z)} \int_{\partial \Sto(w,R_1)} g(\zeta)\frac{\stx_{\bar w,m}(\zeta)}{z-\zeta} \,{d\zeta}\]
and if we take into account that~\eqref{complexsphericalnumber} is independent of the choice of $r$, we conclude that
\[g(z) = \sum_{h=-m}^n  \stx_{w,h}(z) c_h + \mr{F}_n(z) + \mr{F}_{-m}(z)\,.\]
If we set for $t=1,2$
\[M_t :=  \frac{1}{2\pi}\,\underset{\partial \Sto(w,{R_t})}{\max}|g|\ \frac{\mr{length}(\partial \Sto(w,{R_t}))}{ \mr{dist}(z,\partial \Sto(w,{R_t}))}\]
then for all $k\geq 1$
\begin{align*}
\left|F_{2k-1}(z)\right| &\leq M_2 \left(\frac{\sto(z,w)}{{R_2}}\right)^{2k}\,,\\
\left|F_{2k}(z)\right| &\leq M_2 \left(\frac{\sto(z,w)}{{R_2}}\right)^{2k}\frac{\sqrt{\sto(z,w)^2+\beta^2}+\beta}{\sqrt{{R_2}^2+\beta^2}-\beta}\, ,
\end{align*}
where we have applied Lemma~\ref{technical} twice, with $\beta:= |\im(w)|$. Similarly,
\begin{align*}
\left|F_{-2k}(z)\right| &\leq M_1 \left(\frac{{R_1}}{\sto(z,w)}\right)^{2k}\,,\\
\left|F_{-2k-1}(z)\right| &\leq M_1 \left(\frac{{R_1}}{\sto(z,w)}\right)^{2k} \frac{\sqrt{{R_1}^2+\beta^2}+\beta}{\sqrt{\sto(z,w)^2+\beta^2}-\beta}\, ,
\end{align*}
which completes the proof.
\end{proof}

We now turn back to the general case.

\begin{theorem} \label{Laurent-spherical-expansion} 
Let $f: \OO \lra A$ be a slice regular function, let $y \in Q_A$ and let $r_1,r_2 \in [0,+\infty]$ with $r_1 < r_2$ such that $\Sto(y,r_1,r_2) \subseteq \OO$. Then
\begin{equation}\label{sphericalexpansion}
 f(x)=\sum_{n \in \zz}\stx_{y,n}(x)\cdot c_n
\end{equation}
in $\Sto(y,r_1,r_2)$ with
\begin{equation}\label{sphericalnumber}
c_n:=(2\pi J)^{-1}\int_{\partial \Sto_J(y,r)}d\zeta \, (\stx_{y,n+1}(\zeta))^{-1} f(\zeta),
\end{equation}
where $J \in \s_A$ is such that $y \in \cc_J$ and $r$ is in the interval $(r_1,r_2)$. If $r_1=0$, then we call \eqref{sphericalexpansion} \emph{the spherical Laurent expansion of $f$ at $y$} and each coefficient $c_n$ \emph{the $\mathrm{n}^{\mathrm{th}}$-spherical number of $f$ at $y$.} If, moreover, $\Sto(y,r_2)\subseteq \OO$, then $c_n=0$ for all $n<0$ and the expansion is valid in the whole Cassini ball $\Sto(y,r_2)$.
\end{theorem}

\begin{proof}
Let $\{1,J, J_1, J J_1,\ldots, J_h, JJ_h\}$ be a splitting basis of $A$ associated with $J$ and let $\{f_k:\OO_J \lra \cc_J\}_{k=0}^h$ be the holomorphic functions such that $f_{|_{\OO_J}}=\sum_{k=0}^h f_k J_k$, where $J_0:=1$. Applying Lemma \ref{lem:C} to each $f_k$, we obtain that
\begin{equation*}
f_k(z)=\sum_{n \in \zz} \stx_{y,n}(z) c_{n,k}
\end{equation*}
for all $z \in \Sto_J(y,r_1,r_2)$, where each $c_{n,k} \in \cc_J$ is defined by an appropriate instance of formula~\eqref{complexsphericalnumber}. Thus, for all $z \in \Sto_J(y,r_1,r_2)$,
\begin{equation*}
f(z) = \sum_{k=0}^h \sum_{n \in \zz} \left(\stx_{y,n}(z) c_{n,k}\right) J_k = \sum_{k=0}^h \sum_{n \in \zz}\stx_{y,n}(z)  \left(c_{n,k} J_k \right) = \sum_{n \in \zz} \stx_{y,n}(z) c_n
\end{equation*}
with $c_n:=\sum_{k=0}^hc_{n,k}J_k$. In the second equality, we have used the fact that, since $\stx_{y,n}(z),c_{n,k} \in \cc_J$, the associator $(\stx_{y,n}(z),c_{n,k},J_k)=0$ vanishes.
By Lemma~\ref{productpreservedslice} and by Definition~\ref{lineintegral}, the coefficient $c_n$ can also be computed by formula~\eqref{sphericalnumber}. 

By Corollary~\ref{domainconvergence}, the series $\sum_{n \in \zz} \stx_{y,n}(x)\cdot c_n$ converges in $\Sto(y,r_1,r_2)$ to a slice regular function $g$. Finally, since $g$ and $f$ coincide in $\Sto_J(y,r_1,r_2)$, they coincide in 
$\Sto(y,r_1,r_2)$ by the Representation Formula~\eqref{rep2}.
\end{proof}

\begin{remark}
Suppose that the hypotheses of Theorem~\ref{Laurent-spherical-expansion} hold. For each $J \in \s_A$, we have $f(\Sto_J(y,r_1,r_2))\subseteq\cc_J$ if, and only if, $c_n \in \cc_J$ for all $n \in \zz$.
\end{remark}

We point out that the spherical Laurent expansions at a point $y$ and at $y' \in \s_y$ are related, according to Remark~\ref{rearrange}. Moreover, the same remark allows an alternative expansion depending only on the sphere $\s_y$ considered rather than on a specific point $y$:

\begin{remark}\label{rmk:rearrange}
Expansion~\eqref{sphericalexpansion} can be rearranged into the form
\begin{equation} \label{eq:spherical}
f(x)=\sum_{k \in \zz}\Delta_y^k(x)(xu_k+v_k)
\end{equation}
for all $x \in \Sto(y,r_1,r_2)$.  If $r_1=0$, then \eqref{eq:spherical} is called \emph{the spherical Laurent expansion of $f$ at $\s_y$}. Moreover, $(u_k,v_k)$ is called \emph{the $\mr{k}^{\mr{th}}$-spherical pair of $f$ at $\s_y$}. 

For all $k \in \zz$,
\begin{equation} \label{eq:u_n} 
u_k=(2\pi J)^{-1}\int_{\partial \Sto_J(y,r)}d\zeta \, \Delta_y(\zeta)^{-k-1} f(\zeta)
\end{equation}
and
\begin{equation} \label{eq:v_n} 
v_k=(2\pi J)^{-1}\int_{\partial \Sto_J(y,r)}d\zeta \, (\zeta-2\re(y))\, \Delta_y(\zeta)^{-k-1} f(\zeta).
\end{equation}
For each $J \in \s_A$, we have $f(\Sto_J(y,r_1,r_2))\subseteq\cc_J$ if, and only if, $u_k,v_k \in \cc_J$ for all $k \in \zz$. As a consequence, when the cardinality of $\s_A$ is larger than $2$, then $f$ is slice preserving in $\Sto(y,r_1,r_2)$ if, and only if, $u_k,v_k \in \rr$ for all $k \in \zz$.
\end{remark}

We point out that the last statement in the previous remark is actually true for all cardinalities of $\s_A$, as a consequence of the next observation.

\begin{remark}\label{schwarzreflection}
Let $f \in \mc{S}(\OO)$ and fix $J \in \s_A$. The function $f$ is slice preserving if, and only if, the following properties hold:
\begin{enumerate}
\item $f$ maps $\OO_J$ into $\cc_J$;
\item $f$ has the Schwarz reflection symmetry in $\OO_J$, i.e., $f(z^c) = f(z)^c$ for all $z \in \OO_J$.
\end{enumerate}
\end{remark}


\section{Relation between the two expansions}\label{sec:relation}

In the present section, we relate the spherical Laurent expansion at $y$ to the Laurent expansions at $y$ and $y^c$. Even though the resulting combinatorial formulae are quite technical, both the next theorem and the subsequent corollary will prove extremely useful to classify singularities.

\begin{theorem}\label{thm:relation}
Let $f:\OO \lra A$ be a slice regular function and let $y \in Q_A$. Suppose that $\OO$ contains a nonempty Cassini shell $\Sto(y,r_1,r_2)$, which in turn contains nonempty $\Sigma(y,R_1,R_2)$ and $\Sigma(y^c,R_1,R_2)$. Consider the Laurent expansions
\begin{align*}
f(x)&=\sum_{n\in\zz}(x-y)^{\punto n}\cdot a_n,\\
f(x)&=\sum_{n\in\zz}(x-y^c)^{\punto n}\cdot b_n
\end{align*}
and the spherical Laurent expansion
\begin{align*}
f(x)&=\sum_{n\in\zz}\stx_{y,n}(x)\cdot c_n.
\end{align*}
If $y \in \rr$, then $a_n=b_n=c_n$ for all $n \in \zz$. If, on the other hand, $y \in Q_A \setminus \rr$, then for every $m,n \in \zz$, the following relations hold:
\begin{equation}\label{an}
a_n=\sum_{m=-\infty}^{n}(y-y^c)^{2m-n-1}\left(\binom{m-1}{n-m}c_{2m-1} +\binom{m}{n-m}(y-y^c)c_{2m}\right)
\end{equation}
and
\begin{align}
c_{2m}&=\sum_{\ell=-\infty}^m (y-y^c)^{\ell-2m-1}\left( \binom{-m}{m-\ell} (y-y^c) a_\ell + (-1)^{\ell-1}\binom{-m-1}{m-\ell} b_{\ell-1} \right)\label{s2m}\\
c_{2m+1}&=\sum_{\ell=-\infty}^m(y-y^c)^{\ell-2m-1}\binom{-m-1}{m-\ell}\left(a_\ell + (-1)^{\ell-1}b_{\ell} \right).\label{s2m1}
\end{align}
\end{theorem}

\begin{proof}
Using a splitting basis, the desired formulae will be proven if we can prove them in the special case when $\{a_n\}_{n\in \zz},\{b_n\}_{n\in \zz},\{c_n\}_{n\in \zz}$ are all included in $\cc_J$. Equivalently, we may restrict to the case when $A = \cc$ and consider a holomorphic function $f:\Sto(w,r_1,r_2)\rightarrow\cc$ with $w=\alpha+i\beta$. According to our hypotheses, $\Sto(w,r_1,r_2)$ contains two nonempty annuli, centered at $w$ and $\overline w$ respectively.

The equality $f(\xi)=\sum_{\ell\in\zz}\stx_{w,\ell}(\xi) c_\ell$ implies that $a_n=\sum_{\ell\in\zz}d_{n,\ell}c_\ell$ where, for any circle $\gamma$ in $\Sto(w,r_1,r_2)$ centered at $w$, 
\[
d_{n,\ell}:=\frac1{2\pi i}\int_{\gamma}\frac{\stx_{w,\ell}(\xi)}{(\xi-w)^{n+1}}d\xi\,.
\]
We can compute the integrals  $d_{n,\ell}$ by means of the complex residue theorem. 
\begin{itemize}
\item In the even case, $\ell=2m$, we get
\begin{equation}\label{sum1}
\stx_{w,2m}(\xi)=(\xi-w)^m(\xi-\overline w)^m= \sum_{k\ge0}\binom{m}{k}(w-\overline w)^{m-k}(\xi-w)^{m+k}\,,
\end{equation}
whence
\[\frac{\stx_{w,2m}(\xi)}{(\xi-w)^{n+1}} = \sum_{k\ge0}\binom{m}{k}(w-\overline w)^{m-k}(\xi-w)^{m+k-n-1}\,.\]
By the complex residue theorem,
\[
d_{n,2m}:=\binom{m}{n-m}(w-\overline w)^{2m-n}\,
\]
for $m\leq n$ and $d_{n,2m}=0$ otherwise.
\item In the odd case, $\ell=2m-1$, we get that 
\[\frac{\stx_{w,2m-1}(\xi)}{(\xi-w)^{n+1}} =(\xi-w)^{m-n-1}(\xi-\overline w)^{m-1}= \sum_{k\ge0}\binom{m-1}{k}(w-\overline w)^{m-k-1}(\xi-w)^{m+k-n-1}\,.\]
Hence,
\[d_{n,2m-1}=\binom{m-1}{n-m}(w-\overline w)^{2m-n-1}\]
for $m\leq n$ and $d_{n,2m-1}=0$ otherwise.
\end{itemize}
The equality $a_n=\sum_{\ell\in \zz}d_{n,\ell}c_\ell=\sum_{m=-\infty}^{n}(d_{n,2m-1}c_{2m-1} +d_{n,2m}c_{2m})$ yields formula \eqref{an}.
	 
Now we prove \eqref{s2m} and \eqref{s2m1}. We recall from \eqref{sphericalnumber} that 
\[ c_n=(2\pi i)^{-1}\int_{\partial \Sto(w,r)}\frac{f(\xi)}{\stx_{w,n+1}(\xi)} \, d\xi \, ,\]
so that $c_n$ is the sum of the residues of the holomorphic function ${f(\xi)}/{\stx_{w,n+1}(\xi)}$ at $w$ and $\overline w$.
\begin{itemize}
\item For $n=2m$ even, we have
\begin{align*}
\frac {f(\xi)}{\stx_{w,2m+1}(\xi)}&=\frac{\sum_{\ell\in\zz}(\xi- w)^\ell a_\ell}{(\xi-w)^{m+1}(\xi-\overline w)^{m}}\\
&=\sum_{\ell\in\zz}a_\ell \sum_{k\ge0}\binom{-m}k (w-\overline w)^{-m-k}(\xi-w)^{k+\ell-m-1}\,,
\end{align*}
whose residue at $w$ is
\begin{equation}\label{res1}
\sum_{\ell=-\infty}^m \binom{-m}{m-\ell} (w-\overline w)^{\ell-2m} a_\ell.
\end{equation}
Similarly, at $\overline w$ we have the Laurent expansion
\begin{align*}
\frac {f(\xi)}{\stx_{w,2m+1}(\xi)}&=\frac{\sum_{\ell\in\zz}(\xi- \overline w)^\ell b_\ell}{(\xi-w)^{m+1}(\xi-\overline w)^{m}}\\
&=\sum_{\ell\in\zz}b_\ell \sum_{k\ge0}\binom{-m-1}k (\overline w-w)^{-m-k-1}(\xi-\overline w)^{k+\ell-m},
\end{align*}
whose residue at $\overline w$ is 
\begin{equation}\label{res2}
\sum_{\ell=-\infty}^{m-1}\binom{-m-1}{m-\ell-1} (\overline w-w)^{\ell-2m}b_{\ell} = \sum_{\ell=-\infty}^{m}\binom{-m-1}{m-\ell} (\overline w-w)^{\ell-2m-1}b_{\ell-1}.
\end{equation}
Putting together \eqref{res1} and \eqref{res2}, we get \eqref{s2m}.

\item For $n=2m+1$ odd, it holds
\begin{align*}
\frac {f(\xi)}{\stx_{w,2m+2}(\xi)}&=\frac{\sum_{\ell\in\zz}(\xi- w)^\ell a_\ell}{(\xi-w)^{m+1}(\xi-\overline w)^{m+1}}\\
&=\sum_{\ell\in\zz}a_\ell \sum_{k\ge0}\binom{-m-1}k (w-\overline w)^{-m-k-1}(\xi-w)^{k+\ell-m-1}\,,\\
\frac {f(\xi)}{\stx_{w,2m+2}(\xi)}&=\frac{\sum_{\ell\in\zz}(\xi- \overline w)^\ell b_\ell}{(\xi-w)^{m+1}(\xi-\overline w)^{m+1}}\\
&=\sum_{\ell\in\zz}b_\ell \sum_{k\ge0}\binom{-m-1}k (\overline w-w)^{-m-k-1}(\xi-\overline w)^{k+\ell-m-1}\,.
\end{align*}
Therefore the residues at $w$ and  $\overline w$ are, respectively,
\[\sum_{\ell=-\infty}^m\binom{-m-1}{m-\ell} (w-\overline w)^{\ell-2m-1}a_\ell
\text{\quad and\quad}
\sum_{\ell=-\infty}^{m}\binom{-m-1}{m-\ell} (\overline w-w)^{\ell-2m-1}b_{\ell}\,,
\]
which yield \eqref{s2m1}.\qedhere
\end{itemize}
\end{proof}

\begin{corollary}\label{zerocoefficients}
In the hypotheses of Theorem~\ref{thm:relation} and under the assumption $y \in Q_A \setminus \rr$, the following implications hold.
\begin{enumerate}
\item If $c_p = 0$ for all $p<2n_0-1$, then $a_n =0$ for all $n< n_0$.
\item If $a_\ell= 0 = b_\ell$ for all $\ell<m_0$, then $c_p = 0$ for all $p< 2m_0$.
\end{enumerate}
\end{corollary}

\begin{example}
The function $f(x) = (x+J)^{-\punto} = (x^2+1)^{-1}(x-J)$ considered in Example~\ref{ex:expansion} has, at $y=J$
\begin{itemize}
\item $a_n = 0$ for all $n<0$ and $a_n = (-1)^n(2J)^{-n-1}$ for all $n\geq0$.
\item $b_n = 0$ for all $n\neq -1$ and $b_{-1} = 1$.
\item $c_n = 0$ for all $n\neq -1$ and $c_{-1} = 1$.
\end{itemize}
\end{example}


\section{Classification of singularities}\label{sec:classification}

In the present section, we will undertake the classification of singularities by taking advantage of the spherical Laurent series. This is always possible because of the next remark.

\begin{remark}\label{singularityadmitsspherical}
Let $y\in Q_A$. If a circular open set $\OO$ in $Q_A$ includes $\Sigma(y,0,R)$ for some $R>0$, then it includes $\Sto(y,0,r)$ for some $r>0$. Indeed, the continuous function $Q_A \to \rr$ defined by $x \mapsto \sto(x,y)$ vanishes on $\s_y$, it is positive elsewhere and it tends to $+\infty$ as $\|x\|_A \to +\infty$.
\end{remark}

This allows, in particular, the next definition.

\begin{definition}
Let $f: \OO \lra A$ be a slice regular function, let $y \in Q_A$ be a singularity for $f$ and let
\[
f(x)= \sum_{n \in \zz}\stx_{y,n}(x)\cdot c_n = \sum_{k \in \zz}\Delta_y^k(x)(xu_k+v_k),
\]
be the spherical Laurent expansion of $f$ at $y$ and at $\s_y$, respectively. We define the \emph{spherical order} of $f$ at $\s_y$ as the smallest even natural number $n_0$ such that $c_n=0$ for all $n<-n_0$. It is denoted by $\ord_f(\s_y)$. If such an $n_0$ exists, then $k_0=\frac{n_0}{2}$ is the smallest natural number such that $u_k=0=v_k$ for all $k<-k_0$. If no such $n_0$ exists, then we set $\ord_f(\s_y) := +\infty$.
\end{definition}

In the forthcoming Remark~\ref{rmk:semiregulartoquotient}, we will see that this definition is consistent with the one given in the quaternionic case.

\begin{remark} \label{sphericalvsisolatedorder}
Let $f: \OO \lra A$ be a slice regular function and let $y \in Q_A$ be a singularity for $f$. If $y \not\in \rr$, then Corollary~\ref{zerocoefficients} immediately implies that
\[
\ord_f(\s_y)=2\max\big\{\ord_f(y),\ord_f(y^c)\big\}.
\]
If $y \in \rr$, then $\ord_f(\s_y)=\ord_f(y)$ if $\ord_f(y)$ is even and $\ord_f(\s_y)=\ord_f(y)+1$ otherwise.
\end{remark}

We are now ready to classify the singularities of slice regular functions. First, we consider the case of non real singularities. We recall the notation
\[
V(g) := \{x \in \OO : g(x)=0\}
\]
used for $g \in \mc{S}(\OO)$.

\begin{theorem}\label{sphericalclassification}
Let $\widetilde{\OO}$ be a circular open subset of $Q_A$, let $y \in \widetilde{\OO}$ and set $\OO:=\widetilde{\OO}\setminus \s_y$. Let $f: \OO \lra A$ be a slice regular function. Let $J \in \s_A$ be such that $y \in \cc_J\setminus \rr$ and set $\OO_J:=\OO \cap \cc_J$, $\widetilde{\OO}_J:=\widetilde{\OO} \cap \cc_J=\OO_J \cup \{y,y^c\}$, as usual, and $f_J:=f|_{\OO_J}$. Then one of the following assertions holds:
\begin{enumerate}
 \item Every point of $\s_y$ is a removable singularity for $f$.\\ This is equivalent to each of the following conditions:
  \begin{enumerate}
   \item[1a.] $\ord_f(\s_y)=0$.
   \item[1b.] $\ord_f(y)=0=\ord_f(y^c)$.
   \item[1c.] There exists a neighborhood $U_J$ of $\{y,y^c\}$ in $\widetilde{\OO}_J$ such that $\|f_J\|_A$ is bounded in $U_J \setminus \{y,y^c\}$.
   \item[1d.] There exists a neighborhood $U$ of $\s_y$ in $\widetilde{\OO}$ such that $\|f\|_A$ is bounded in $U \setminus \s_y$.
  \end{enumerate}
 \item Every point of $\s_y$ is a non removable pole for $f$.\\ This is equivalent to each of the following conditions:
  \begin{enumerate}
   \item[2a.] $\ord_f(\s_y)$ is finite and positive.
   \item[2b.] $\ord_f(y),\ord_f(y^c)$ are finite and at least one of them is positive.
   \item[2c.] There exists $k \in \nn \setminus \{0\}$ such that the function $\OO_J \lra A$ defined by 
   \[z \mapsto (z-y)^k(z-y^c)^kf_J(z)\]
   extends to a continuous function $g_J:\widetilde{\OO}_J \lra A$ with $g_J(y) \neq 0$ (which yields $\lim_{\OO_J \ni x \to y}\|f(x)\|_A = +\infty$) or $g_J(y^c) \neq 0$ (whence $\lim_{\OO_J \ni x \to y^c}\|f(x)\|_A=+\infty$).
   \item[2d.] There exists $k \in \nn \setminus \{0\}$ such that the function $\OO \lra A$ defined by
   \[x \mapsto \Delta_y(x)^kf(x)\]
   extends to a slice regular function $g \in \mc{SR}(\widetilde{\OO})$ that does not vanish identically in $\s_y$. For all $w \in \s_y \setminus V(g)$, we have that $2k=\ord_f(\s_y)=2\,\ord_f(w)$ and
\[
\lim_{\OO \ni x \to w} \|f(x)\|_A = +\infty.
\]
For all $w \in \s_y \cap V(g)$, the strict inequality $\ord_f(w)<k$ holds.
  \end{enumerate}
 \item For every $I \in \s_A$, the intersection $\s_y \cap \cc_I$ includes an essential singularity for~$f$.\\ This is equivalent to each of the following conditions: \begin{enumerate}
   \item[3a.] $\ord_f(\s_y)=+\infty$.
   \item[3b.] $\ord_f(y),\ord_f(y^c)$ are not both finite.
   \item[3c.] For all neighborhoods $U_J$ of $y$ in $\widetilde{\OO}_J$ (with $y^c \not \in U_J)$ and for all $k \in \nn$, 
    \[\sup_{z \in U_J \setminus \{y\}}\|(z-y)^kf(z)\|_A=+\infty;\]
   or the sentence is true with $y$ and $y^c$ swapped.
\item[3d.] For all neighborhoods $U$ of $\s_y$ in $\widetilde{\OO}$ and for all $k \in \nn$, \[\sup_{x \in U \setminus \s_y}\|\Delta_y(x)^kf(x)\|_A=+\infty.\]
 \end{enumerate}
\end{enumerate}
\end{theorem}

\begin{proof} 
We divide the proof into several steps.
\begin{itemize}
\item We first consider the case when $\ord_f(\s_y)$ is finite. In this case, we let $k:=\frac{1}{2}\ord_f(\s_y)$ and define a slice regular function $G:\OO \lra A$ by setting $G(x):=\Delta_y^k(x) f(x)$. By considering the spherical Laurent expansions of $f$ at $y$ and at $\s_y$,
\[
f(x)=\sum_{n=-2k}^{+\infty}\stx_{y,n}(x)\cdot c_n=\sum_{n=-k}^{+\infty}\Delta_y^n(x)(xu_n+v_n),
\]
valid in $\Sto(y,0,R)$ for some $R>0$, we conclude that
\[
G(x)=\sum_{m=0}^{+\infty}\stx_{y,m}(x)\cdot c_{m-2k}=\sum_{m=0}^{+\infty}\Delta_y^m(x)(xu_{m-k}+v_{m-k})
\]
in $\Sto(y,0,R)$. By Theorem~\ref{thm:Abel-spherical}, the domains of convergence of the last two series include $\Sto(y,R)$. Therefore, $G$ extends to a slice regular function $g \in \mc{SR}(\widetilde{\OO})$. In particular, considering that $\s_y$ is compact, $G=g_{|_\OO}$ is bounded locally at $\s_y$. 

If $k=0$ then $g$ is an extension of $f$ to $\widetilde{\OO}$ and all points of $\s_y$ are removable singularities for $f$.

If $k>0$ then we can additionally observe that $g$ does not vanish identically on $\s_y$ (whence $g(y)\neq0$ or $g(y^c)\neq0$). Indeed,
\[g(x) = xu_{-k}+v_{-k}\]
for all $x \in \s_y$ and an equality $u_{-k}=0=v_{-k}$ would contradict $\ord_f(\s_y)=2k$. Since $\|f(x)\|_A \geq c_A|\Delta_y(x)|^{-k}\|g(x)\|_A$ for every $x \in \OO$, if $w \in \s_y \setminus V(g)$, then
\[
\lim_{\OO \ni x \to w}\|f(x)\|_A=+\infty.
\]
We now claim that $\ord_f(w)=k$ for all $w \in \s_y \setminus V(g)$ and $\ord_f(w)<k$ for all $w \in \s_y \cap V(g)$.
For every $w \in \s_y$, denote by $f(x)=\sum_{n \in \zz}(x-w)^{\punto n}\cdot a_n(w)$ the Laurent expansion of $f$ at $w$. Now, $c_\ell(w)=0$ for all $\ell<-2k$ and $c_{-2k}(w)=wu_{-k}+v_{-k}$ thanks to formulae~\eqref{eq:uk} and \eqref{eq:vk}. By applying formula \eqref{an}, we conclude that $a_n(w)=0$ for every $n<-k$ and $a_{-k}(w)=(w-w^c)^{-k}(wu_{-k}+v_{-k})$. As a consequence, we have that $a_{-k}(w)=(w-w^c)^{-k}g(w)$ and the claim is proven.

Altogether, taking into account Remark~\ref{sphericalvsisolatedorder}, we conclude that 
\[\mathit{1b} \Rightarrow \mathit{1a} \Rightarrow \mathit{1} \Rightarrow \mathit{1d} \Rightarrow \mathit{1c} \,,\]
\[\mathit{2b} \Rightarrow \mathit{2a} \Rightarrow \mathit{2} \Rightarrow \mathit{2d} \Rightarrow \mathit{2c}\,.\]

\item We now prove that
\[\mathit{3c} \Rightarrow \mathit{3d} \Rightarrow \mathit{3} \Rightarrow \mathit{3a} \Rightarrow \mathit{3b}\,.\]
Indeed, if {\it 3c} holds then the inequality 
\[\|\Delta_y(z)^kf(z)\|_A\geq c_A |\Delta_y(z)|^k\|f(z)\|_A=c_A |z-y|^k|z-y^c|^k\|f(z)\|_A,\]
valid for all $z \in U_J \setminus\{y,y^c\}$, allows to deduce {\it 3d}.

If {\it 3d} holds then there cannot exist $I \in \s$ such that the two points of $\s_y \cap \cc_I$ are both poles: otherwise, {\it 1d} or {\it 2d} would hold and contradict {\it 3d}. Hence, {\it 3} holds.

If {\it 3} holds then $\ord_f(\s_y)$ cannot be finite, otherwise {\it 1} or {\it 2} would hold and contradict {\it 3}. Thus, {\it 3a} holds.

Finally, if {\it 3a} holds then {\it 3b} follows by Remark~\ref{sphericalvsisolatedorder}.

 \item Let us close the three cycles of implications by proving that $\mathit{1c}\Rightarrow\mathit{1b}$, $\mathit{2c}\Rightarrow\mathit{2b}$ and $\mathit{3b}\Rightarrow\mathit{3c}$. To do so, we consider a splitting basis $\{1,J,J_1,JJ_1,\ldots,J_h,JJ_h\}$ of $A$ associated with $J$ and denote by $f_0,f_1,\ldots,f_h$ the holomorphic functions from $\OO_J$ to $\cc_J$ such that $f|_{\OO_J}=\sum_{\ell=0}^hf_\ell J_\ell$. 
By direct inspection in the proof of Theorem~\ref{Laurent}, $\ord_f(y)=\max_{\ell \in \{0,1,\ldots,h\}}\ord_{f_\ell}(y) \in \nn \cup \{+\infty\}$. 

Property {\it 1c} holds if, and only if, each holomorphic function $f_\ell$ is bounded in $U_J\setminus\{y,y^c\}$. By a well-known fact in classical complex analysis, this is equivalent to saying that $\ord_{f_\ell}(y)=0=\ord_{f_\ell}(y^c)$ for all $\ell \in \{0,1,\ldots,h\}$. This is, in turn, equivalent to property {\it 1b}.

If {\it 2c} holds then, for each $\ell \in \{0,1,\ldots,h\}$, the function $(z-y)^k(z-y^c)^kf_\ell(z)$ extends continuously, hence holomorphically, to $\widetilde{\OO}_J$. Thus, $y$ and $y^c$ are poles with $\ord_{f_\ell}(y),\ord_{f_\ell}(y^c) \leq k$. As a consequence, $\ord_{f}(y),\ord_{f}(y^c)$ are finite and they cannot both be zero, otherwise {\it 1b} would hold, which we proved equivalent to {\it 1c}. This would contradict {\it 2c}.

If {\it 3b} holds then there exists $\ell \in \{0,1,\ldots,h\}$ such that $\ord_{f_\ell}(y)=+\infty$ or $\ord_{f_\ell}(y^c)=+\infty$. Without loss of generality, the former equality holds. Thus, in all neighborhoods $U_J$ of $y$ in $\widetilde{\OO}_J$ and for all $k \in \nn$, $\sup_{z \in U_J \setminus \{y,y^c\}}|z-y|^k|f_\ell(z)|=+\infty$, whence {\it 3c} follows.

\item Since exactly one among properties {\it 1a}, {\it 2a}, {\it 3a} holds, the proof is complete.
\end{itemize}
\end{proof}

A completely analogous argument proves the next statement.

\begin{theorem}\label{sphericalclassificationreal}
Let $\widetilde{\OO}$ be a circular open subset of $Q_A$, let $y \in \widetilde{\OO} \cap \rr$ and set $\OO:=\widetilde{\OO}\setminus \{y\}$. Let $f: \OO \lra A$ be a slice regular function. Then one of the following properties holds:
\begin{enumerate}
 \item The point $y$ is a removable singularity for $f$.\\
 This is equivalent to each of the following conditions:
  \begin{itemize}
   \item[1a.] $\ord_f(y)=0$.
    \item[1b.] There exists a neighborhood $U$ of $y$ in $\widetilde{\OO}$ such that $\|f\|_A$ is bounded in $U \setminus \{y\}$.
  \end{itemize}
 \item The point $y$ is a non removable pole for $f$.\\
 This is equivalent to each of the following conditions:
 \begin{itemize}
   \item[2a.] $\ord_f(y)$ is finite and positive.
   \item[2b.]  There exists $k \in \nn \setminus \{0\}$ such that the function $\OO \to A$ defined by 
   \[x \mapsto (x-y)^kf(x)\]
   extends to a slice regular function $g \in \mc{SR}(\widetilde{\OO})$ with $g(y)\neq0$. In this case, $k=\ord_f(y)$ and
$\lim_{\OO \ni x \to y} \|f(x)\|_A = +\infty$.
  \end{itemize}
\item The point $y$ is an essential singularity for $f$.\\
 This is equivalent to each of the following conditions:
  \begin{itemize}
   \item[3a.] $\ord_f(y)=+\infty$
   \item[3b.] For all neighborhoods $U$ of $y$ in $\widetilde{\OO}$ and for all $k \in \nn$, 
   \[\sup_{x \in U \setminus \{y\}}\|(x-y)^kf(x)\|_A=+\infty.\]
   \end{itemize}
\end{enumerate}
\end{theorem}

Theorem~\ref{sphericalclassificationreal} classifies real singularities of slice regular functions exactly as those of holomorphic functions of one complex variable. On the other hand, for singularities that do not lie on the real axis, Theorem~\ref{sphericalclassification} shows a different panorama. Already in the quaternionic case~\cite{singularities}, it turned out that on each sphere of poles, the order is constant with the possible exception of one point (which must have lesser order). The same happens in the octonionic case. Here is a different example over the $8$-dimensional Clifford algebra $\rr_3$:

\begin{example}
Let $A=\rr_3$ and set $\OO:=Q_A \setminus \s_A$. If $f: \OO \lra A$ is the slice regular function defined by $f(x):= (x+e_1)^{-\punto}\cdot(x-e_{23})$, then $f(x) = (x^2+1)^{-1} g(x)$ with
\[
g(x):=(x-e_1)\cdot(x-e_{23}).
\]
It is known that $V(g) \cap \s_A = \{e_1,e_{23}\}$ (see~\cite[Proposition 4.10]{gpsalgebra}). Therefore, $\ord_f(y)=1$ and 
\[\lim_{\OO \ni x \to y} \|f(x)\|_A = +\infty\]
for all $y \in \s_A \setminus \{e_1,e_{23}\}$. On the other hand, $\ord_f(e_1)=0=\ord_f(e_{23})$ and 
\[\lim_{\OO \ni x \to e_1} \|f(x)\|_A, \lim_{\OO \ni x \to e_{23}} \|f(x)\|_A\]
are not defined. Indeed, there exist finite
\[\lim_{\cc_{e_1} \ni x \to e_1} f(x),\quad \lim_{\cc_{e_{23}} \ni x \to e_{23}} f(x)\]
but $ \|f(x)\|_A$ is unbounded in every neighborhood $U$ of $e_1$ or $e_{23}$ in $Q_A$, because $e_1$ and $e_{23}$ are not isolated points in $\s_A$.
\end{example}

Over a general algebra $A$, it is not a trivial matter to study the order function on a sphere $\s_y$. The next section is devoted to this problem.


\section{Study of the order function}\label{sec:order}

We would now like to understand how much the order can vary on a sphere $\s_y$ of singularities for $f$. It is clear from Theorem~\ref{sphericalclassification} that this problem is related to the zero set $V(g)$ appearing in case {\it 2d}. Therefore, we begin with some new results concerning the zeros of slice functions.

We point out that the forthcoming theorem and the subsequent remark apply even without the Assumption (F) taken in the present paper. We will use both the Euclidean topology and the Zariski topology. We refer the reader to \cite{bochnakcosteroy} for more details on this and other notions from real algebraic geometry mentioned henceforth.

\begin{theorem}\label{thm:zeros}
Let $g:\OO \lra \hh$ be a slice function and let $y \in V(g) \setminus \rr$. Then there exists a real affine subspace $H$ of $A$ through $y$ such that
\[
V(g) \cap \s_y = H \cap \s_y.
\]
Now let us endow $\s_y$ with either the Euclidean or the Zariski relative topology. If $W$ denotes a connected subset of $\s_y$ containing $y$, then one of the following properties holds.
\begin{enumerate}
\item $g$ vanishes identically in $\s_y$.
\item $g$ vanishes identically in $W$ and it vanishes nowhere in $W^c$.
\item There exists a nonempty open subset $\widetilde W$ of $W$ such that $g$ vanishes nowhere in $\widetilde W \cup \widetilde W^c$.
\end{enumerate}
If $g'_s(y)=0$ then property {\it 1} holds. If, on the other hand, $g'_s(y)\neq0$, then either property {\it 2} or {\it 3} holds and the former is excluded under any of the following additional hypotheses:
\begin{itemize}
\item[$(a)$] $W$ is preserved by $^*$-involution; that is, $W^c=W$.
\item[$(b)$] $W$ intersects $W^c$.
\item[$(c)$] There exist distinct $y',y'' \in W$ such that $y'-y''$ is not a left zero divisor.
\item[$(d)$] There exist $y',y'' \in W$ such that $(y'-y'')g'_s(y)\neq0$.
\end{itemize}
\end{theorem}

\begin{proof}
Since $g$ is affine when restricted to $\s_y$, the first statement immediately follows. Now let us consider the second statement. The intersection $W \cap V(g)$ is a closed subset of $W$ while the difference $W \setminus V(g)$ is an open subset of $W$.

If there exists $x \in W$ such that $g(x)\neq0\neq g(x^c)$, then property {\it 3} holds.

If, for some $x \in W$, $g(x)=0=g(x^c)$ then $g$ vanishes identically in $\s_y$ by the Representation Formula~\eqref{rep2} and property {\it 1} holds.

Let us now focus on the other possible case: for all $x \in W$ either $g(x)=0$ or $g(x^c)=0$.
Since $x \mapsto x^c$ is a homeomorphism from $\s_y$ to itself, $W^c$ is a connected subset of $\s_y$ containing $y^c$. If we set $U:= W \cup W^c$, then $U\cap V(g)$ is closed in $U$, $U\setminus V(g)$ is open in $U$ and they are swapped by $^*$-involution. Therefore, in $U:= W \cup W^c$, the subset $U \cap V(g)$ is open and closed and it intersects $W$ at $y$; the subset $U \setminus V(g)$ is open and closed and it intersects $W^c$ at $y^c$. This implies that $U\cap V(g)=W$, while $U\setminus V(g)=W^c$, which is property {\it 2}.

As for our final statements, $g'_s(y)=0$ implies $g(y)=0=g(y^c)$ by definition. Moreover, $(a)\Rightarrow(b)\Rightarrow(c)\Rightarrow(d)$ where $(d)$ implies that 
\[
g(y') - g(y'') =(\im(y')-\im(y''))g'_s(y) = (y'-y'')g'_s(y)\neq0\,,
\]
so that $g(y')$ and $g(y'')$ cannot both be $0$.
\end{proof}

A significant example of case {\it 2} can be constructed over the bicomplex numbers.

\begin{example}
We saw in Example~\ref{bicomplex} that $\s_{\B\cc}$ consists of four points, namely $e^+$, $e^-$, $-e^+$, $-e^-$. The polynomial
\[f(x):=x^2-x(e^++e^-)+e^+e^- = (x-e^+)\cdot(x-e^-)=(x-e^-)\cdot(x-e^+)\]
vanishes in $W:=\{e^+\}$ and in $Z:=\{e^-\}$ but not in $W^c=\{-e^+\}$ nor in $Z^c=\{-e^-\}$.
\end{example}

\begin{remark}
It is worth noting that the statement of Theorem \ref{thm:zeros} with respect to the Zariski topology is stronger than the same statement with respect to the Euclidean topology. Indeed, the Zariski topology is coarser than the Euclidean one and hence, if a subset of $\s_y$ is connected with respect to the Euclidean topology, then it is also Zariski-connected. In general, the converse implication is false. 

A real algebraic set $V$ is connected with respect to the Zariski topology if either it is irreducible or it is reducible and, for every pair of distinct irreducible components $M$, $N$ of $V$, there exists a finite sequence $(V_1,\ldots,V_k)$ of irreducible components of $V$, with $k \geq 2$, such that $M=V_1$, $N=V_k$ and $V_i \cap V_{i+1} \neq \emptyset$ for every $i \in \{1,\ldots,k-1\}$.
\end{remark}

In the Zariski topology, we can draw from Theorem~\ref{thm:zeros} the following consequence.

\begin{corollary}\label{cor:zeros}
Let $g:\OO \lra \hh$ be a slice function and let $y\in V(g)$ be such that $g$ does not vanish identically in $\s_y$. If $Y$ is an irreducible component of $\s_y$ including $y$, then either $g$ vanishes identically in $Y$ and nowhere in $Y^c$ or $g$ is nonzero at a generic point of $Y\cup Y^c$; that is, $\dim_\rr((Y \cup Y^c) \cap V(g))<\dim_\rr(Y \cup Y^c)=\dim_\rr(Y)$.
\end{corollary}

\begin{remark}
In the division algebras $\hh$ and $\oo$, the $\s_y$ are Euclidean spheres. As a consequence, the unique irreducible component of each $\s_y$ is the whole sphere $\s_y$ itself, which is invariant under $^*$-involution. Thus, if $g \in \mc{S}(\OO)$, if $y\in V(g)$ and if $g$ does not vanish identically in $\s_y$ then $\dim_\rr(\s_y \cap V(g))<\dim_\rr(\s_y)$. 
\end{remark}

This is consistent with the detailed study of the zeros conducted specifically for the quaternions, see~\cite[Theorem 3.1]{librospringer}, and for the octonions, see~\cite[Theorem 1]{ghiloni}.
Moreover, Corollary~\ref{cor:zeros} yields a new result concerning split quaternions.

\begin{remark}
In the algebra of split quaternions $\s\hh$, the $\s_y$ are two-sheet hyperboloids of real dimension $2$, see~\cite[Example 1.13]{gpsalgebra}. By~\cite[Theorem 4.5.1]{bochnakcosteroy}, the unique irreducible component of each $\s_y$ is the whole $\s_y$ itself, which is invariant under $^*$-involution. As a consequence,  if $g \in \mc{S}(\OO)$, if $y\in V(g)$ and if $g$ does not vanish identically in $\s_y$ then $\dim_\rr(\s_y \cap V(g))\leq 1$. 
\end{remark}

Before proceeding to study the order function on a sphere $\s_y$, let us prove a useful lemma. 

\begin{lemma} \label{lem:polesonasphere}
Let $f:\OO \lra A$ be a slice regular function, let $y \in Q_A$ be a singularity for $f$ and, for every $n \in \zz$, let $a_n:\s_y \lra A$ be the function which associates to each $w \in \s_y$ the $n^{\mr{th}}$-coefficient $a_n(w)$ of the Laurent expansion of $f$ at $w$. Denote by $P$ be the set of all poles of $f$ belonging to $\s_y$. Then there exist $m,\ell \in \nn$ (depending on $f$ and $\s_y$) with $1\leq m \leq \ell$ such that
\[P=\bigcap_{n=-\ell}^{-m}\,\{w \in \s_y \, | \, a_n(w)=0\}.\] 
In particular, there exists a real affine subspace $H$ of $A$ such that $P=\s_y \cap H$.
\end{lemma}

\begin{proof}
By Corollary \ref{coefficientsareslicefunctions}, there exist $b_n,c_n \in A$ such that $a_n(w)=wb_n+c_n$. If $H_n$ denotes the affine subspace $\{x \in A \, | \, xb_n+c_n=0\}$ of $A$, then we have that $P=\s_y \cap H$, where 
\[H:=\bigcup_{k \in \nn}\bigcap_{n<-k}H_n.\]
Since $\{\bigcap_{n<-k}H_n\}_{k \in \nn}$ is a non-decreasing sequence of affine subspaces of $A$, there exists $m \geq 1$ such that 
\[H=\bigcap_{n \leq -m}H_n=\bigcap_{k \geq m}\bigcap_{n=-k}^{-m}H_n.\]
Now, the sequence $\{\bigcap_{n=-k}^{-m}H_n\}_{k \geq m}$ of affine subspaces of $A$ is non-increasing. It follows that $H=\bigcap_{n=-\ell}^{-m}H_n$ for some $\ell \geq m$.  
\end{proof}

We are now ready to study the order function.

\begin{theorem}
Let $f: \OO \lra A$ be a slice regular function and let $y \in Q_A \setminus \rr$ be a singularity for $f$. If $2\,\ord_f(y)<\ord_f(\s_y)$ then there exists a real affine subspace $H$ of $A$ through $y$ such that
\[
\{x \in \s_y\ |\ 2\,\ord_f(x)<\ord_f(\s_y)\} = \s_y \cap H.
\]
If $W$ denotes a Zariski (or Euclidean) connected subset of $\s_y$ containing $y$, then one of the following properties holds.
\begin{enumerate}
\item For all $x \in W$, $2\,\ord_f(x)<\ord_f(\s_y)=2\,\ord_f(x^c)$.
\item There exists a nonempty Zariski (or Euclidean) open subset $\widetilde W$ of $W$ such that $2\,\ord_f(x)=\ord_f(\s_y)=2\,\ord_f(x^c)$ for all $\widetilde W$.
\end{enumerate}
Case {\it 1} is excluded under any of the following additional assumptions: $(a)$ $W$ is preserved by $^*$-involution; $(b)$ $W$ intersects $W^c$; $(c)$ there exist distinct $y',y'' \in W$ such that $y'-y''$ is not a left zero divisor; or $(d)$ there exists $y' \in W$ such that $2\,\ord_f(y')=\ord_f(\s_y)$.

Finally, if $Y$ is an irreducible component of $\s_y$ containing $y$, then one of the following properties holds.
\begin{enumerate}
\item[1'.] For all $x \in Y$, $2\,\ord_f(x)<\ord_f(\s_y)=2\,\ord_f(x^c)$.
\item[2'.] There exists a real algebraic subset $Y'$ of $Y$ such that $\dim_\rr(Y')<\dim_\rr(Y)$ and $2\,\ord_f(x)=2\,\ord_f(x^c)=\ord_f(\s_y)$ for every $x \in Y \setminus Y'$.
\end{enumerate}
\end{theorem}

\begin{proof}
If $\ord_f(\s_y)$ is finite, then our statements follows directly from case {\it 2d} in Theorem~\ref{sphericalclassification}, along with Theorem~\ref{thm:zeros} and Corollary~\ref{cor:zeros}. If, on the other hand, $\ord_f(\s_y)=+\infty$ then they follow from Theorem~\ref{thm:zeros}, Corollary~\ref{cor:zeros} and Lemma~\ref{lem:polesonasphere}.
\end{proof}


\section{The algebra of slice semiregular functions}\label{sec:algebra}

We conclude this paper with a study of the analogs of meromorphic functions.

\begin{definition}\label{def:semiregular}
A function $f$ is \emph{(slice) semiregular} in a (nonempty) circular open subset $\OO$ of $Q_A$ if there exists a circular open subset $\OO'$ of $\OO$ such that $f \in \mc{SR}(\OO')$ and such that every point of $\OO \setminus \OO'$ is a pole (or a removable singularity) for $f$.
\end{definition}

As an application of Remark~\ref{rmk:rearrange}, as well as Theorems~\ref{sphericalclassification} and~\ref{sphericalclassificationreal}, we make a new observation.

\begin{remark}\label{rmk:semiregulartoquotient}
Let $f$ be a semiregular function on a circular open subset $\OO$ of $Q_A$. For each $y \in \OO$, there exists $r>0$ such that $f$ is slice regular in $\Sto(y,0,r)$. Either $f \equiv 0$ in $\Sto(y,0,r)$ or one of the following properties holds.
\begin{itemize}
\item  If $y \not \in \rr$ then there exists $m \in \zz$ such that, for $x \in \Sto(y,0,r)$,
\[f(x) = \Delta_y^m(x) g(x)\]
for some $g \in \mc{SR}(\Sto(y,r))$ that does not vanish identically in $\s_y$.  If $m<0$ then the spherical order $\ord_f(\s_y)$ equals $-2m$. If, on the other hand, $m\geq0$ then $\ord_f(\s_y)=0$ and $2m$ will be called the \emph{spherical multiplicity of $f$ at $\s_y$.} 
\item If $y \in \rr$ then there exists $n \in \zz$ such that, for $x \in \Sto(y,0,r)= Q_A \cap B(y,r)\setminus\{y\}$,
\[f(x) = (x-y)^n h(x)\]
for some $h \in \mc{SR}(\Sto(y,r))$ with $h(y) \neq 0$. Moreover, $f(x) = \Delta_y^m(x) g(x)$ with $m:=\left\lfloor \frac n2\right\rfloor$ and $g \in \mc{SR}(\Sto(y,r))$. If $n<0$ then $\ord_f(y)=-n$ and $ \ord_f(\s_y)=-2m$. If, on the other hand, $n\geq0$ then $\ord_f(y)=0=\ord_f(\s_y)$ and $n$ will be called the \emph{classical multiplicity of $f$ at $y$}. 
\end{itemize}
Conventionally, if $y \not \in \rr$ then the spherical multiplicity of $f$ at $\s_y$ is set to $0$ if $m<0$ and to $+\infty$ if $f\equiv0$ in $\Sto(y,0,r)$. Similarly, if $y \in \rr$ then the classical multiplicity of $f$ at $y$ is set to $0$ if $n<0$ and to $+\infty$ if $f\equiv0$ in $\Sto(y,0,r) = Q_A \cap B(y,r)\setminus\{y\}$.
\end{remark}

In particular, a semiregular function can be locally expressed as a quotient of slice regular functions. This enables us to study the algebraic properties of semiregular functions.

\begin{theorem} \label{thm:semiregular}
Let $\OO$ be a a circular open subset of $Q_A$. The set of semiregular functions on $\OO$ is a real alternative algebra with respect to $+,\cdot$. It becomes a $^*$-algebra when endowed with the involution $f \mapsto f^c$. It is nonsingular if, and only if, $A$ is nonsingular and $\OO$ is a union of slice domains.
\end{theorem}

\begin{proof}
Let $f_1,f_2$ be semiregular in $\OO$. Let $\OO'$ be obtained from $\OO$ by erasing the nonremovable singularities of $f_1$ and let $\OO''$ be obtained from $\OO'$ by erasing the nonremovable singularities of $f_2$. By~Proposition~\ref{prop:algebra}, $f_1+f_2, f_1\cdot f_2$ and $f_1^c$ are elements of $\mc{SR}(\OO'')$ and we are left with proving that every $y \in \OO \setminus\OO''$ is a pole for all of them. If $k_1:=\frac{1}{2}\,\ord_{f_1}(\s_y)$ and $k_2:=\frac{1}{2}\,\ord_{f_2}(\s_y)$ then, for some $r>0$ and for $x \in \Sto(y,0,r)$,
\[f_1(x) = \Delta_y^{-k_1}(x)g_1(x),\quad f_2(x) = \Delta_y^{-k_2}(x) g_2(x)\]
with $g_1,g_2 \in \mc{SR}(\Sto(y,r))$. Thus,
\begin{align*}
&f_1(x)+f_2(x) = \Delta_y^{-\max\{k_1,k_2\}}(x)\,h(x)\\
&(f_1\cdot f_2)(x) = \Delta_y^{-k_1-k_2}(x)\,p(x)\\
&f_1^c(x) = \Delta_y^{-k_1}(x)\,g_1^c(x)
\end{align*}
with $h,p,g_1^c \in \mc{SR}(\Sto(y,r))$. By Theorem \ref{sphericalclassification}, $h,p,g_1^c$ have spherical order $0$ at $\s_y$. By Remark~\ref{rmk:semiregulartoquotient}, 
\begin{align*}
&h(x) = \Delta_y^{\ell}(x)\,q(x)\\
&p(x) = \Delta_y^{m}(x)\,r(x)\\
&g_1^c(x) = \Delta_y^{n}(x)\,s(x)
\end{align*}
with $\ell,m,n \geq 0$ and with $q,r,s \in \mc{SR}(\Sto(y,r))$ that do not vanish identically in $\s_y$. By a further application of Remark~\ref{rmk:semiregulartoquotient},
we conclude that the spherical orders of $f+g,f\cdot g$ and $f^c$ at $\s_y$ are finite. The thesis follows by Theorem~\ref{sphericalclassification}.
Our final statement follows from Proposition~\ref{SRnonsingular}.
\end{proof}

Our final remarks concern division among semiregular functions. We begin with a definition, which extends Definition~\ref{def:tame}.

\begin{definition}\label{def:tamesemiregular}
In the hypotheses of Definition~\ref{def:semiregular}, $f$ is said to be \emph{tame} in $\OO$ if it is tame in $\OO'$.
\end{definition}

\begin{remark}
If $f$ is semiregular and tame in a circular open subset $\OO$ of $Q_A$, then:
\begin{itemize}
\item for all $y \in \OO \setminus \rr$ with $\ord_{N(f)}(\s_y)=2k$, there exists $r>0$ such that in $\Sto(y,r)$ the function $g(x):=\Delta_y^k(x)\,N(f)(x)$ is slice preserving and it coincides with the function $\Delta_y^k(x)\,N(f^c)(x)$;
\item for all $y \in \OO \cap \rr$ with $\ord_{N(f)}(y)=n$, there exists $r>0$ such that in $\Sto(y,r)=Q_A \cap B(y,r)$ the function $h(x):=(x-y)^n\,N(f)(x)$ is slice preserving and it coincides with $(x-y)^n\,N(f^c)(x)$.
\end{itemize}
Using the notations of Definition~\ref{def:semiregular}, these facts are guaranteed by Definition~\ref{def:tamesemiregular} for all $y\in\OO'$ (where $\ord_{N(f)}(\s_y)=0=\ord_{N(f)}(y)$). For $y \in \OO \setminus\OO'$, they follow by continuity.
\end{remark}

Within the algebra of semiregular functions, division by tame elements is possible thanks to the next result.

\begin{theorem}\label{thm:quotienttosemiregular}
If $f$ is semiregular and tame in a circular open subset $\OO$ of $Q_A$ and if $N(f)\not\equiv 0$ then $f^{-\punto}$ is semiregular and tame in $\OO$.
\end{theorem}

\begin{proof}
Let $\OO'$ be obtained from $\OO$ by erasing the singularities of $f$ that are not removable. According to Proposition~\ref{reciprocal},
\[f^{-\punto} = N(f)^{-\punto} \cdot f^c = N(f)^{-1} f^c\]
is slice regular in $\OO'':=\OO' \setminus V(N(f))$. If $\OO'=\OO_E$ then $V(N(f))$ is the circularization of a closed and discrete subset of $E$. As a consequence, every $y \in \OO\setminus\OO''$ is a singularity for $f^{-\punto}$. We now prove that it is a pole.

By Remark~\ref{rmk:semiregulartoquotient} and by the tameness of $f$, if $y \not \in \rr$ then there exist $m \in \zz$ and $r>0$ such that in $\Sto(y,r)$
\[N(f) = \Delta_y^{\punto m}\cdot g\]
for some $g\in\mc{SR}(\Sto(y,r))$ that does not vanish identically in $\s_y$ and is slice preserving. Hence, $g$ never vanishes in $\s_y$. The multiplicative inverse $g^{-\punto}$ is thus slice preserving and slice regular near $\s_y$. As a consequence, in a neighborhood of $\s_y$
\[f^{-\punto} = (g^{-\punto} \cdot \Delta_y^{-\punto m}) \cdot f^c = \Delta_y^{-\punto m} \cdot g^{-\punto} \cdot f^c\,,\]
where the function $g^{-\punto} \cdot f^c$ is slice regular. It follows that $f^{-\punto}$ has spherical order at $\s_y$ less than or equal to $\max\{0,2m\}$, which is the spherical multiplicity of $N(f)$ at $\s_y$. By Theorem~\ref{sphericalclassification}, $y$ is a pole for $f$.

Similarly, for any $y \in \rr$ it turns out that $\ord_{f^{-\punto}}(y)$ is less than or equal to the classical multiplicity of $N(f)$ at $y$, so that $y$ is a pole for $f$.

Let us now prove that $f^{-\punto}$ is tame. From Theorem~\ref{invertibles_f}, it follows that in $\OO''$
\[N(f^{-\punto})=N(f^c)^{-\punto}=N(f)^{-\punto}=N((f^{-\punto})^c),
\]
where $N(f)^{-\punto}$ is slice preserving because $N(f)$ is. We conclude that $f^{-\punto}$ is tame in $\OO''$, whence in $\OO$.
\end{proof}

We complete our study of the algebra of semiregular functions with a stronger result, which holds in the nonsingular case.

\begin{theorem}\label{thm:groupoftame}
Suppose $A$ is nonsingular and $\OO$ is a slice domain. Within the nonsingular $^*$-algebra of semiregular functions on $\OO$, let $T$ denote the set of all tame elements that are not identically $0$. Then $T$ is a multiplicative Moufang loop and it is preserved by $^*$-involution.
\end{theorem}

\begin{proof}
$T$ is preserved by $^*$-involution because tameness is preserved by it and because $f\equiv0$ is equivalent to $f^c\equiv0$. We will now prove that $T$ is closed under multiplication and then prove that every element of $T$ admits a multiplicative inverse in $T$.

Let $f,g \in T$. We know from Theorem~\ref{thm:semiregular} that $f\cdot g$ is semiregular in $\OO$ and we need to prove that it does not vanish identically and that it is tame. Since $\OO$ is a slice domain, so is the subset $\OO'$ obtained from it by erasing the poles that are not removable for $f$ or for $g$. Since $A$ is nonsingular and $f,g \in \mc{SR}(\OO')$, by Proposition~\ref{SRnonsingular}, $N(f)\not \equiv 0\not \equiv N(g)$ in $\OO'$. According to Proposition~\ref{tamenotzerodivisor}, $f\cdot g$ is not identically $0$ in $\OO'$, thus in $\OO$. Moreover, by Proposition~\ref{tameregularproduct}, $f\cdot g$ is tame in $\OO'$, whence in $\OO$.

Let us now consider $f^{-\punto}$, which is semiregular and tame in $\OO$ by Theorem~\ref{thm:quotienttosemiregular}. Now, $f^{-\punto}$ cannot vanish identically in $\OO$ or $f\cdot f^{-\punto} \equiv 1$ would be contradicted. We conclude that $f^{-\punto}\in T$, as desired.
\end{proof}



\end{document}